\documentclass[12pt,a4paper]{amsart}
\usepackage[top=1.5in, bottom=1.25in, left=1.25in, right=1.25in]{geometry}
\usepackage{amsmath}
\usepackage{amssymb}
\usepackage{amsfonts}
\usepackage{palatino, mathrsfs}
\usepackage{url}
\usepackage{xcolor}
\usepackage{fullpage, microtype, colonequals}

\newtheorem{theorem}{Theorem}[section]
\newtheorem{definition}{Definition}
\newtheorem{lemma}[theorem]{Lemma}
\newtheorem{proposition}[theorem]{Proposition}

\newtheorem*{conjecture*}{Conjecture}
\newtheorem*{theorem*}{Theorem}
\newtheorem*{remark*}{Remark}
\numberwithin{equation}{section}

\usepackage{mathtools}

\begin{document}

\begin{center}

\title[Ergodicity for 3D SNSEs with Markov switching]{Ergodicity for Three-Dimensional Stochastic Navier-Stokes Equations with Markov Switching}

\author{Po-Han Hsu$^\ast$}
\address{4415 French Hall-West, University of Cincinnati, Cincinnati, OH,  45221-0025, USA}
\email{hsupa@ucmail.uc.edu}

\author{Padmanabhan Sundar}
\address{316 Lockett Hall, Louisiana State University, Baton Rouge, LA, 70803-4918,  USA}
\email{psundar@lsu.edu}

\date{\today}

\thanks{$\ast$The author  was partially supported by NSF grant DMS-1622026.}

\begin{abstract}
Asymptotic behavior of the three-dimensional stochastic Navier-Stokes equations with Markov switching in additive noises is studied for incompressible fluid flow in a bounded domain in the three-dimensional space. To study such a system, we introduce a family of regularized equations and investigate the asymptotic behavior of the regularized equations first. The existence an ergodic measure for the regularized system is established via the Krylov-Bogolyubov method. Then the existence of an stationary measure to the original system is obtained by extracting a limit from the ergodic measures of the family of the regularized system.
\end{abstract}
\maketitle
\end{center}

\section{Introduction}
Let $G$ be an open bounded domain in $\mathbb{R}^{3}$ with a smooth boundary. Let the three-dimensional vector-valued function ${\bf u}(x, t)$ and  the real-valued function $p(x, t)$ denote the velocity and pressure of the fluid at each $x\in G$ and time $t\in [0, T]$. The motion of viscous incompressible flow on $G$ with no slip at the boundary is described by the Navier-Stokes system:
\begin{align}
\partial_{t}{\bf u}-\nu\Delta{\bf u}+({\bf u}\cdot\nabla){\bf u}-\nabla p&={\bf f}(t)\quad&&\mbox{in}\quad G\times[0, T] \label{eq},
\\ \nabla\cdot {\bf u}&=0\quad&&\mbox{in}\quad G\times[0, T],\nonumber
\\ {\bf u}(x,t)&=0 \quad&&\mbox{on}\quad\partial G\times[0, T],\nonumber
\\ {\bf u}(x, 0)&={\bf u}_{0}(x)\quad&&\mbox{on}\quad G\times\{t=0\},\nonumber
\end{align}
where $\nu>0$ denotes the viscosity coefficient, and the function ${\bf f}(t)$ is an external body force.  Recalling the Helmholtz decomposition, which states that $L^{2}(G)$ can be written as a direct sum of solenoidal part and irrotational part, and applying the Leray projector to equation \eqref{eq}, one may write equation \eqref{eq} in the abstract evolution form on a suitable space as follows (see, e.g., \cite{Leray, Temam} for details):
\begin{align}\label{evolution B}
{\bf du}(t)+[\nu{\bf Au}(t)+{\bf B}({\bf u}(t))]dt={\bf f}(t)dt,
\end{align}
where ${\bf A}$ is the Stokes operator and ${\bf B}$ is the nonlinear inertial operator  introduced in Section \ref{pre}. 

A random body force, in the form of a additive noise driven by a Wiener process  $W(t)$, is added to the model (see, e.g., \cite{turbulence}) so that one obtains
\begin{align*}
\mathbf{du}(t)+[\nu\mathbf{Au}(t)+\mathbf{B}(\mathbf{u}(t))]dt=\mathbf{f}(t)dt+\sigma(t)dW(t).
\end{align*}
Moreover, if the noise is allowed to be ``discontinuous,'' then a term driven by a Poisson random measure $N_{1}(dz, ds)$ (which is independent of $W(t)$) is added so that the equation becomes
 \begin{align}\label{eq tur}
\mathbf{du}(t)+[\nu\mathbf{Au}(t)+\mathbf{B}(\mathbf{u}(t))]dt
=\mathbf{f}(t)dt+\sigma(t)dW(t)+\int_{Z}{\bf G}(t, z)\tilde{N_{1}}(dz, dt),
\end{align}
where $\tilde{N_{1}}(dz, dt):=N_{1}(dz, dt)-\nu_{1}(dz)dt$ and $\nu_{1}(dz)dt$ is the intensity measure of $N_{1}(dz, dt)$. 

Let $m>0$ be a fixed integer, $\mathcal{S}=\{1, 2,\cdots, m\}$, and $\{\mathfrak{r}(t):t\in\mathbb{R}^{+}\}$ be an ergodic right continuous Markov chain taking values in $\mathcal{S}$. The following equation whose noise terms depend on the Markov chain $\mathfrak{r}(t)$ allows for transition in the type of  random forces that perturb the Navier-Stokes equation:
\begin{align}\label{equation 0 MS paper}
\begin{split}
&{\bf du}(t)+[\nu{\bf Au}(t)+{\bf B}({\bf u}(t))]dt
\\&={\bf f}(t)dt+\sigma(\mathfrak{r}(t))dW(t)+\int_{Z}{\bf G}(\mathfrak{r}(t-), z)\tilde{N}_{1}(dz, dt)
\end{split}
\end{align}
and is called the stochastic Navier-Stokes equation with Markov switching.

The stochastic Navier-Stokes equation with Markov switching was introduced in the earlier work of the authors \cite{SNSEs markov}, and the existence of a  weak solution (in the sense of both partial differential equations and stochastic analysis) was obtained under suitable hypotheses for multiplicative noises. 

The objective of the present article is to study the asymptotic behavior of equation \eqref{equation 0 MS paper}. Originally, it was Kolmogorov's idea to introduce a Wiener process to the right side of the equation \eqref{eq} in order to investigate the existence of invariant measures (see, e.g., \cite{Vishik}). From then on, several works on the study of invariant measures of the two-dimensional stochastic Navier-Stokes equations appeared with an additive noise driven by a Wiener process (see, e.g., \cite{Albeverio, Flandoli, Flandoli and Maslowski, FG, Mattingly, Sundar measure} ). For the stochastic Navier-Stokes in three-dimensional space, the study of the invariant measure is slight different from the case in two-dimensional space. Due to the lack of uniqueness of the solution, it is hard to have a well-defined transition probability to the system (see, e.g., \cite{FG}). However, in a work of Da Prato and Debussche \cite{DD}, they constructed a transition semigroup that admits a unique invariant measure for the three-dimensional stochastic Navier-Stokes equations driven by Wiener process. Later, the technique was adapted by Mohan, Sakthivel, and Sritharan \cite{MSS} to construct a transition semigroup that admits a unique invariant measure for the three-dimensional stochastic Navier-Stokes equations with L{\'e}vy noise. For more results on the ergodicity of the stochastic Navier-Stokes equations, we refer the interested reader to the survey article by Debussche \cite{Debussche}.

The introduction of a Markov chain into the noise terms of a stochastic system may be traced back to Skorohod \cite{Skorohod}. In \cite{Skorohod}, Skorohod introduced a Markov chain in the the noise term of a stochastic system and then studied ergodic behaviors of the resulting system with small noise. Later, several results regarding the ergodicity of stochastic differential equations with Markov switching appeared in literature (see, e.g., \cite{Mao Yuan, Mao, MYY}). The interested reader may consult the books by Mao and Yuan \cite{Mao book} and Yin and Zhu \cite{hybrid}. 

A novelty in our work consists in the introduction of Markov chain in order to allow for transitions in the types of random forces that perturb the Navier-Stokes system. In addition, the system under consideration is a three-dimensional system. The nonlinearity in such a case is less regular. To overcome the difficulties caused by the nonlinearity of the system, we follow Leray's idea \cite{Leray} to regularized the nonlinear term in \eqref{equation 0 MS paper}:
\begin{align}\label{equation 1 MS paper}
\begin{split}
&{\bf du}^{\epsilon}(t)+[\nu{\bf Au}^{\epsilon}(t)+{\bf B}_{k_{\epsilon}}({\bf u}^{\epsilon}(t))]dt
\\&={\bf f}(t)dt+\sigma(\mathfrak{r}(t))dW(t)+\int_{Z}{\bf G}(\mathfrak{r}(t-), z)\tilde{N}_{1}(dz, dt)
\end{split}
\end{align}
for each $\epsilon>0$, and refer it the regularized equation (the detailed definition of ${\bf B}_{k_{\epsilon}}$ will be given in Section \ref{pre}). For the regularized equation, we showed that it will ultimately approach to a steady state, which we refer to the exponential stability (Theorem \ref{exponential stability MS paper}). Next, by employing the method of Krylov-Bogolyubov, we constructed a stationary measure to the regularized system \eqref{equation 1 MS paper}, which together with Theorem \ref{exponential stability MS paper} leads the uniqueness. Moreover, by \cite[Theorem 3.2.6]{DZ}, we see that such a measure is indeed ergodic (Theorem \ref{existence of stationary MS paper}).

Let $\{{\bf u}^{\epsilon}\}_{\epsilon>0}$ be a family of solutions to equation \eqref{equation 1 MS paper}. It is shown in \cite{SNSEs markov} that there is a subsequence $\epsilon_{k}$ of $\epsilon$ such that ${\bf u}^{\epsilon_{k}}\rightarrow{\bf u}$ weakly in the ``path space'' (see equation \eqref{the path space} below), as $k\rightarrow\infty$, and the function ${\bf u}$ is a solution to equation \eqref{equation 0 MS paper}. Denote by $\lambda^{\epsilon}$ the ergodic measure induced by ${\bf u}^{\epsilon}$. Then we showed in Theorem \ref{existence of stationary 0 MS paper} that
\begin{enumerate}
\item the solution ${\bf u}$ induces a stationary measure $\lambda$ to the system \eqref{equation 0 MS paper}, and
\item $\lambda^{\epsilon_{\ell}}\rightarrow\lambda$, as $\ell\rightarrow \infty$, where $\epsilon_{\ell}$ is a further subsequence of $\epsilon_{k}$.
\end{enumerate}

The rest of the article is organized as follows. Prerequisites and functional analytic setup, the hypotheses for additive noises, and essential results in \cite{SNSEs markov} will be introduced and recalled in Section \ref{pre}. Section \ref{regular} is devoted to the study of the regularized system \eqref{equation 1 MS paper}. We will deduce a priori estimates to the regularized system \eqref{equation 1 MS paper}. Then we present the result of exponential stability (Theorem \ref{exponential stability MS paper}) and construct the unique ergodic measure to the regularized system (Theorem \ref{existence of stationary MS paper}). 
In Section \ref{original}, we prove that equation \eqref{equation 0 MS paper} admits a stationary measure, and such a stationary measure is a limit of the family of ergodic measures of the regularized system (Theorem \ref{existence of stationary 0 MS paper}).

\section{Prerequisites and Functional Analytic Setup}\label{pre}
\subsection{Basic Results on Convolution}
First, we recall some properties of convolution in order to explain regularization. The interested reader may consult, e.g., \cite[Appendix C.5.]{Evans} for more details. If $U\subset\mathbb{R}^{3}$ is open and $\epsilon>0$, we write
$
U_{\epsilon}:=\{x\in U: \text{dist}(x, \partial U)>\epsilon\}.
$
Define the function $\eta\in C^{\infty}(\mathbb{R}^{3})$ by
\begin{align*}
\eta(x):=
\begin{cases}
C\exp\Big(\frac{1}{|x|^{2}-1}\Big)& \mbox{if}\quad |x|<1\\
0 & \mbox{if}\quad |x|\geq 1,
\end{cases}
\end{align*}
where the constant $C>0$ is selected so that $\int_{\mathbb{R}^{3}}\eta dx=1$. For each $\epsilon>0$, set
\begin{align*}
\eta_{\epsilon}(x):=\frac{1}{\epsilon^{3}}\eta\Big(\frac{x}{\epsilon}\Big).
\end{align*}
We call $\eta$ the standard mollifier. The function $\eta_{\epsilon}$ is smooth with support in $B(0, \epsilon)$ and satisfy
$
\int_{\mathbb{R}^{3}}\eta_{\epsilon}dx=1.
$
If $f: U\rightarrow\mathbb{R}$ is locally integrable, define the mollification operator by 
\begin{align}\label{k_{epsilon}}
k_{\epsilon}f:=\eta_{\epsilon}\ast f \quad\mbox{in} \quad U_{\epsilon},
\end{align}
i.e., 
$
k_{\epsilon}f=\int_{U}\eta_{\epsilon}(x-y)f(y)dy=\int_{B(0, \epsilon)}\eta(y)f(x-y)dy
$
for $x\in U_{\epsilon}$.
The next lemma collects some properties of the mollification operator. The interested reader may consult, e.g., \cite[Theorem 7 in Appendix C.5]{Evans} or \cite[Lemma 6.3]{O-P} for details.
\begin{lemma}\label{convolution}
The mollification operator enjoys the following properties:
\begin{enumerate}
\item $k_{\epsilon}f\in C^{\infty}(U_{\epsilon})$.
\item If $1\leq p<\infty$ and $f\in L^{P}_{loc}(U)$, then $k_{\epsilon}f\rightarrow f$ in $L^{p}_{loc}(U)$.
\item If $1\leq p<\infty$ and $f\in L^{p}_{loc}(U)$, then $\|k_{\epsilon}f\|_{L^{p}_{loc}(U)}\leq \|f\|_{L^{p}_{loc}(U)}$.
\end{enumerate}
\end{lemma}

\subsection{Function Space and Operators}
Let $G\subset\mathbb{R}^{3}$ be a bounded domain with smooth boundary, $\mathcal{D}(G)$ be the space of $C^{\infty}$-functions with compact support contained in $G$, and $\mathcal{V}\colonequals\{{\bf u}\in\mathcal{D}(G): \nabla\cdot{\bf u}=0\}$. Let $H$ and $V$ be the completion of $\mathcal{V}$ in $L^{2}(G)$ and $W^{1, 2}_{0}(G)$, respectively. Then it can be shown that (see, e.g., \cite[Sec. 1.4, Ch. I]{Temam})
\begin{align*}
H&=\{{\bf u}^{}\in L^{2}(G): \nabla\cdot{\bf u}^{}=0,\ {\bf u}^{}\cdot{\bf n}\big|_{\partial G}=0\},\\
V&=\{{\bf u}^{}\in W^{1, 2}_{0}(G):  \nabla\cdot{\bf u}^{}=0\},
\end{align*}
and we denote the $H$-norm  ($V$-norm, resp.) by $|\cdot|$ ($\|\cdot\|$, resp.) and the inner product on $H$ (the inner product on $V$, resp.) by $(\cdot, \cdot)$ ($(\!(\cdot,\cdot)\!)$, resp.). The duality paring between $V'$ and $V$ is denoted by $\langle\cdot, \cdot\rangle_{V}$ or simply $\langle\cdot, \cdot\rangle$ when there is no ambiguity. In addition, we have the following inclusion between the spaces:
$
V\hookrightarrow H\hookrightarrow V',
$
and both of the inclusions $V\hookrightarrow H$ and $H\hookrightarrow V'$ are compact embeddings ( see, e.g., \cite[Lemma 1.5.1 and 1.5.2, Ch. II]{sohr}).

Let $\bf A: V\rightarrow V'$ be the Stokes operator 
and $\lambda_{1}$ be the first eigenvalue of $\bf A$. Then the Poincar\'e inequality in the context of this article appears:
\begin{align}\label{Poincare ineq}
\lambda_{1}|{\bf u}|^{2}\leq\|{\bf u}\|^{2}
\end{align}
for all ${\bf u}\in V$ (see \cite[Eq. (5.11), Ch. II]{FMRT}).

Define 
$b(\cdot,\cdot,\cdot): V\times V\times V\rightarrow\mathbb{R}$ by
\begin{align*}
b({\bf u},{\bf v},{\bf w}):=\sum_{i,j=1}^{3}\int_{G}{u}_i\frac{\partial v_j}{\partial x_i}w_j dx.
\end{align*}
Then $b$ is a trilinear form (see, e.g., \cite[Sec. 1 Ch. II]{Temam}) which induces a bilinear form ${\bf B}({\bf u,v})$ by $b({\bf u},{\bf v,w})=\langle {\bf B}{\bf (u,v),w}\rangle$.

For each $\epsilon>0$, define  
$b(k_{\epsilon}\cdot,\cdot,\cdot): V\times V\times V\rightarrow\mathbb{R}$ by
\begin{align*}
b(k_{\epsilon}{\bf u},{\bf v},{\bf w}):=\sum_{i,j=1}^{3}\int_{G}{(\eta_{\epsilon}\ast u)}_i\frac{\partial v_j}{\partial x_i}w_j dx,
\end{align*}
which induces a bilinear form ${\bf B}_{k_{\epsilon}}({\bf u,v})$ by $b(k_{\epsilon}{\bf u},{\bf v,w})=\langle {\bf B}_{k_{\epsilon}}{\bf (u,v),w}\rangle$. The regularization rises the regularity of the first component in $b$, therefore, one may employ the  (generalized) H\"older inequality and the Young convolution inequality to deduce
\begin{align}
|b(k_{\epsilon}{\bf u}, {\bf v}, {\bf w})|
\leq \|\eta_{\epsilon}\ast {\bf u}\|_{6}\|\nabla {\bf v}\|_{2}\|{\bf w}\|_{3}
\leq \|\eta_{\epsilon}\|_{\frac{6}{5}}\|{\bf u}\|_{3}\|\nabla {\bf v}\|_{2}\|{\bf w}\|_{3}\label{ineq of b},
\end{align}
which together with  Sobolev embedding and interpolation inequalities further implies
\begin{align}
|b(k_{\epsilon}{\bf u,v,w})|
\leq C_{\epsilon}\|{\bf u}\|^{\frac{1}{2}}|{\bf u}|^{\frac{1}{2}}\|{\bf v}\| \|{\bf w}\|^{\frac{1}{2}}|{\bf w}|^{\frac{1}{2}}\label{b_{k}uvw}.
\end{align}
In particular, when ${\bf u=w}$, we have
\begin{align}
|b(k_{\epsilon}{\bf u,v,u})|\leq C_{\epsilon}\|{\bf u}\|\cdot|{\bf u}|\cdot\|{\bf v}\|.\label{b_{k} uvu}
\end{align}
As shall be seen later, we will work with a fixed $\epsilon$ for regularized equations, therefore, we shall assume that $C_{\epsilon}=1$ for the sake of simplicity.

Denoted by $\{\tau_{i}\}_{i=1}^{4}$ the topologies
\begin{alignat*}{2}
\tau_{1}&=\mbox{$J$-topology}\quad  &&\mbox{on} \quad \mathcal{D}([0, T]; V'),\\
\tau_{2}&=\mbox{weak topology}\quad &&\mbox{on}\quad L^{2}(0, T; V),\\
\tau_{3}&=\mbox{weak-star topology}\quad &&\mbox{on}\quad L^{\infty}(0, T; H),\\
\tau_{4}&=\mbox{strong topology}\quad&&\mbox{on}\quad L^{2}(0, T; H),
\end{alignat*}
and $\Omega_{i}$ the spaces
\begin{align*}
\Omega_{1}&=\mathcal{D}([0, T]; V'),\\
\Omega_{2}&=L^{2}(0, T; V),\\
\Omega_{3}&=L^{\infty}(0, T; H),\\
\Omega_{4}&=L^{2}(0, T; H).
\end{align*}
Then $\{(\Omega_{i}, \tau_{i})\}_{i=1}^{4}$ are all Lusin spaces (a topological space that is homeomorphic to a Borel set of a Polish space).

\begin{definition}\label{the path space of u}
Define the space $\Omega^{\ast}$ by
\begin{align*}
\Omega^{*}=\cap_{i=1}^{4}\Omega_{i}.
\end{align*}
Let $\tau$ be the supremum of the topologies\footnote{The coarest topology that is finer than each $\tau_{i}$. See, e.g., \cite[Sec. 5.2]{Howes}} induced on $\Omega^{\ast}$ by all $\tau_{i}$. Then it follows from a result of Metivier \cite[Proposition 1, Ch. IV]{Metivier} that\footnote{Note that all the natural inclusion $\Omega_{i}\hookrightarrow\Omega_{1}$, $i=2, 3, 4$, are continuous.}
\begin{enumerate}
\item $(\Omega^{\ast}, \tau)$ is a Lusin space.
\item Let $\{\mu_{k}\}_{k\in\mathbb{N}}$ be a sequence of Borel probability laws on $\Omega^{\ast}$ (on the Borel $\sigma$-algebra $\mathcal{B}(\tau)$) such that their images $\{\mu^{i}_{k}\}_{k\in\mathbb{N}}$ on $(\Omega_{i}, \mathcal{B}(\tau_{i}))$ are tight for $\tau_{i}$ for all $i$. Then $\{\mu_{k}\}_{k\in\mathbb{N}}$ is tight for $\tau$. 
\end{enumerate}
\end{definition}
Let $(\Omega, \mathcal{F}, \mathcal{P})$ be a (complete) probability space on which the following are defined:
\begin{enumerate}
\item $W=\{W(t): 0\leq t\leq T\}$, an $H$-valued $Q$-Wiener process.
\item $N=\{N(z, t): 0\leq t\leq T\quad\mbox{and}\quad z\in Z\}$, the Poisson random measure.
\item $\mathfrak{r}=\{\mathfrak{r}(t): 0\leq t\leq T\}$, the Markov chain.
\item $\xi$, an $H$-valued random variable.
\end{enumerate}
Assume that $\xi$, $W$, $N$, and $\mathfrak{r}(t)$ are mutually independent. For each $t$, define the $\sigma$-field
\begin{align*}
\mathcal{F}_{t}:=\sigma(\xi, \mathfrak{r}(t), W(s), N(z, s): z\in Z, 0\leq s\leq t)\ \cup\ \{\text{all $\mathcal{P}$-null sets in $\mathcal{F}$}\}.
\end{align*} 
Then it is clear that $(\mathcal{F}_{t})$ satisfies the usual conditions, and both $W(t)$ and $N(z, t)$ are $\mathcal{F}_{t}$-adapted processes.

Denote by $\mathcal{J}$ the $J$-topology in the space $\mathcal{D}([0, T]; S)$. Then the path space of the solution to equations \eqref{equation 0 MS paper} and \eqref{equation 1 MS paper} is the following space $\Omega^{\dagger}$ equipped with the topology $\tau^{\dagger}$.
\begin{align}\label{the path space}
\begin{split}
\Omega^{\dagger}&:=\Omega^{\ast}\times\mathcal{D}([0, T]; \mathcal{S}),\\
\tau^{\dagger}&:=\tau\times\mathcal{J}.
\end{split}
\end{align}

\subsection{Noise Terms}
\begin{enumerate}
\item[(i)] Let $Q\in\mathcal{L}(H)$ be a nonnegative, symmetric, trace-class operator. Define $H_{0}\colonequals Q^{\frac{1}{2}}(H)$ with the inner product given by 
$
(u, v)_{0}:=(Q^{-\frac{1}{2}}u, Q^{-\frac{1}{2}}v)_{H}
$
for $u, v\in H_{0}$, where $Q^{-\frac{1}{2}}$ is the inverse of $Q$.
Then it follows from \cite[Proposition C.0.3 (i)]{concise} that $(H_{0}, (\cdot, \cdot)_{0})$ is again a separable Hilbert space. 
Let $\mathcal{L}_{2}(H_{0}, H)$ denote the separable Hilbert space of the Hilbert-Schmidt operators from $H_{0}$ to $H$. Then it can be shown that (see, e.g., \cite[p. 27]{concise})
$
\|L\|_{\mathcal{L}_{2}(H_{0}, H)}=\|L\circ Q^{\frac{1}{2}}\|_{\mathcal{L}_{2}(H, H)}
$
for each $L\in \mathcal{L}_{2}(H_{0}, H)$. Moreover, we write 
$
\|L\|_{L_{Q}}=\|L\|_{\mathcal{L}_{2}(H_{0}, H)}
$
for simplicity.

Let $T>0$ be a fixed real number and $(\Omega, \mathcal{F}, \{\mathcal{F}_{t}\}_{0 \le t\le T}, \mathcal{P})$ be a filtered probability space.Let $W$ be an $H$-valued Wiener process with covariance $Q$. Let $\sigma: [0, T]\times\Omega\rightarrow\mathcal{L}_{2}(H_{0}, H)$ be jointly measurable and adapted. If  $\mathbb{E}\int^{T}_{0}\|\sigma(s)\|^{2}_{\mathcal{L}_{2}(H_{0}, H)}ds<\infty$, then for $t\in [0, T]$, the stochastic integral
$
\int^{t}_{0}\sigma(s)dW(s)
$
is well-defined and is an $H$-valued continuous square integrable martingale.

\item[(ii)] Let $({Z}, \mathcal{B}({ Z}))$ be a measurable space, ${\bf M}$ be the collection of all of nonnegative integer-valued measures on $({Z}, \mathcal{B}({ Z}))$, and $\mathcal{B}({\bf M})$ be the smallest $\sigma$-field on ${\bf M}$ with respect to which all
$
\eta\mapsto \eta(B)
$
are measurable, where $\eta\in{\bf M}$, $\eta(B)\in\mathbb{Z}^{+}\cup\{\infty\}$, and $B\in\mathcal{B}({ Z})$. Let
$N: \Omega\rightarrow {\bf M}$ be a Poisson random measure with intensity measure $\nu$. 

For a Poisson random measure $N(dz, ds)$,  $\tilde{N}(dz, ds)\colonequals N(dz, ds)-\nu(dz)ds$ defines its compensation. Then it can be shown that (see, e.g, \cite[Sec. 3, Ch. II]{I-W}) $\tilde{N}(dz, ds)$ is a square integrable martingale, and for predictable $f$ such that
\begin{align*}
&\mathbb{E}\int^{t+}_{0}\int_{Z}|f(\cdot, z, s)|\nu(dz)ds<\infty,\,\,\text{then}\\
&\int^{t+}_{0}\int_{Z}f(\cdot, z, s)\tilde{N}(dz, ds)
\\&=\int^{t+}_{0}\int_{Z}f(\cdot, z, s)N(dz, ds)-\int^{t}_{0}\int_{Z}f(\cdot, z, s)\nu(dz)ds
\end{align*}
 is a well-defined $\mathcal{F}_{t}$-martingale.

\item[(iii)] Let $m\in\mathbb{N}$. Let $\{\mathfrak{r}(t): t\in\mathbb{R}^{+}\}$ be a right continuous  ergodic  Markov chain with generator $\Gamma=(\gamma_{ij})_{m\times m}$ taking values in $\mathcal{S}:=\{1, 2, 3, .....m\}$ such that
\begin{align*}
\mathcal{R}_{t}(i, j)&=\mathcal{R}(\mathfrak{r}(t+h)=j|\mathfrak{r}(t)=i )
\\&=\left\{\begin{array}{rcl}
\gamma_{ij}h+o(h)&\mbox{if}& i\neq j,\\
1+\gamma_{ii}h+o(h)&\mbox{if}&i= j,\\
\text{and}\,\,\gamma_{ii}&=-\sum_{i\neq j}\gamma_{ij}.
\end{array}\right.
\end{align*}
The transition probability $\mathcal{R}_{t}(i, j)$ satisfies the Chapman-Kolmogorov equation:
\begin{align*}
\mathcal{R}_{t+s}(i, k)=\sum_{j=1}^{m}\mathcal{R}_{s}(i, j)\mathcal{R}_{t}(j, k).
\end{align*}
The Markov chain $\mathfrak{r}(t)$ is assumed to be ergodic, therefore, there exists a stationary distribution $\pi=(\pi_{1}, \cdots, \pi_{m})$ for this Markov chain $\mathfrak{r}(t)$, where $\pi_{j}$ satisfies
\begin{align*}
\lim_{t\rightarrow\infty}\mathcal{R}_{t}(i, j)=\pi_{j}
\end{align*}

In addition, $\mathfrak{r}(t)$ admits the following stochastic integral representation  (see, e.g, \cite[Sec. 2.1, Ch. 2]{Skorohod}):
Let $\Delta_{ij}$ be consecutive, left closed, right open intervals of the real line each having length $\gamma_{ij}$ such that
\begin{align*}
\Delta_{12}&=[0, \gamma_{12}),\ 
\Delta_{13}=[\gamma_{12}, \gamma_{12}+\gamma_{13}),\cdots\\
\Delta_{1m}&=\Big[\sum_{j=2}^{m-1}\gamma_{1j}, \sum_{j=2}^{m}\gamma_{1j}\Big),\cdots\\
\Delta_{2m}&=\Big[\sum_{j=2}^{m}\gamma_{1j}+\sum_{j=1, j\neq 2}^{m-1}\gamma_{2j}, \sum_{j=2}^{m}\gamma_{1j}+\sum_{j=1, j\neq 2}^{m}\gamma_{2j}\Big)
\end{align*}
and so on. Define a function
$
h:\mathcal{S}\times\mathbb{R}\rightarrow\mathbb{R}
$
by
\begin{align*}
h( i, y)=
\begin{cases}
j-i & \mbox{if} \quad y\in\Delta_{ij},\\
0& \mbox{otherwise}.
\end{cases}
\end{align*}
Then
$
d\mathfrak{r}(t)=\int_{\mathbb{R}}h(\mathfrak{r}(t-), y)N_{2}(dt, dy),
$
with initial condition $\mathfrak{r}(0)=\mathfrak{r}_{0}$, where $N_{2}(dt, dy)$ is a Poisson random measure with intensity measure $dt\times\mathfrak{L}(dy)$, in which $\mathfrak{L}$ is the Lebesgue measure on $\mathbb{R}$.

We assume that such a Markov chain, Wiener process, and the Poisson random measure are independent.

\end{enumerate}

\subsection{Hypotheses and Essential Results of Existence Theorems}
In this section, we introduce the hypotheses for the noise terms and recall some essential results regarding the solution to equations \eqref{equation 0 MS paper} and \eqref{equation 1 MS paper}.

Throughout this article,  the functions $\sigma:\mathcal{S}\rightarrow \mathcal{L}_{2}(H_{0}, H)$ and ${\bf G}:\mathcal{S}\times Z\rightarrow H$ assumed to satisfy the following Hypotheses $\bf A$: for any $i\in\mathcal{S}$, there exist a constant $K>0$ such that
\begin{enumerate}
\item [{\bf A1.}]  $\|\sigma(i)\|^{2}_{L_{Q}}\leq K$ and
 
\item [{\bf A2.}]
$\int_{Z}|\mathbf{G}(i, z)|^{p}\nu(dz)\leq K$ for $p=1, 2$, and $4$.

\end{enumerate}
It is clear that the Hypotheses ${\bf A}$ is a subclass of the Hypothesis ${\bf H}$ in \cite{SNSEs markov}. Therefore, equation \eqref{equation 1 MS paper} admits a unique strong solution, and equation \eqref{equation 0 MS paper} admits a weak solution. To be precise, we state the following existence theorem for the benefit of the reader.

\begin{theorem}\label{existence theorem}
Assume  $\mathbb{E}(|{\bf u}_{0}|^{3})<\infty$ and ${\bf f}\in L^{3}(0, T; V')$. Then under the Hypotheses ${\bf A}$, 
\begin{enumerate}
\item equation \eqref{equation 1 MS paper} admits a unique strong solution for each $\epsilon>0$;
\item equation \eqref{equation 0 MS paper} admits a weak solution.
\end{enumerate}
\end{theorem}

\section{The Regularized Equation}\label{regular}
In this section, we study the ergodic properties of the regularized equation \eqref{equation 1 MS paper} with ${\bf u}^{\epsilon}(0)={\bf u}^{\epsilon}_{0}=k_{\epsilon}{\bf u}_{0}$ being an $H$-valued random variable, where $k_{\epsilon}\cdot$ is the mollification operator defined in \eqref{k_{epsilon}}. The next proposition gives a priori estimates for the solution ${\bf u}^{\epsilon}$ under Hypotheses $\bf A$ . It is worth mentioning that the exponent of the time parameter $T$ appears in the upper bounds is 1 rather than any higher power (cf. \cite[Proposition 3.1]{SNSEs markov}), which is crucial in using the Krylov–Bogolyubov method.

\begin{proposition}[A priori estimates]\label{enhanced estimates}
Let  $\epsilon>0$ and $T>0$ be fixed. Assume  $\mathbb{E}(|{\bf u}_{0}|^{3})<\infty$ and ${\bf f}\in L^{3}(0, T; V')$. Then the solution ${\bf u}^{\epsilon}$ of equation \eqref{equation 1 MS paper} satisfies the following estimates.
\begin{align}\label{L^2 strong MS paper}
\mathbb{E}|{\bf u}^{\epsilon}(t)|^{2}+\nu\mathbb{E}\int^{t}_{0}\|{\bf u}^{\epsilon}(s)\|^{2}ds 
\leq \mathbb{E}|{\bf u}_{0}|^{2}+\frac{1}{\nu}\mathbb{E}\int^{t}_{0}\|{\bf f}(s)\|^{2}_{V'}ds+2Kt
\end{align}
for each $t\in (0, T]$, and
\begin{align}\label{L^2 sup strong MS paper}
\mathbb{E}\sup_{t\in [0, T]}|{\bf u}^{\epsilon}(t)|^{2}+\nu\mathbb{E}\int^{T}_{0}\|{\bf u}^{\epsilon}(s)\|^{2}ds\leq 2\mathbb{E}|{\bf u}_{0}|^{2}+\frac{4}{3\nu}\mathbb{E}\int^{T}_{0}\|{\bf f}(s)\|^{2}ds+100KT.
\end{align}
\end{proposition}

\begin{proof}
Let $N>0$. Define 
\begin{align*}
\tau_{N}:=\inf\{t\in[0, T]:&|{\bf u}^{\epsilon}(t)|^{2}+\int^{t}_{0}\|{\bf u}^{\epsilon}(s)\|^{2}ds>N  
\\ \mbox{or}\ &|{\bf u}^{\epsilon}(t-)|^{2}+\int^{t}_{0}\|{\bf u}^{\epsilon}(s)\|^{2}ds>N\}.
\end{align*}
Then for each $t\in[0, T]$, the It\^o formula implies
\begin{align}
&|{\bf u}^{\epsilon}(\tau_{N}\wedge t)|^{2}+2\nu\int^{\tau_{N}\wedge t}_{0}\|{\bf u}^{\epsilon}(s)\|^{2}ds\label{Ito MS MS paper}
\\&=|{\bf u}^{\epsilon}(0)|^{2}+2\int^{\tau_{N}\wedge t}_{0}\langle{\bf f}(s), {\bf u}^{\epsilon}(s)\rangle ds+\int^{\tau_{N}\wedge t}_{0}\|\sigma(\mathfrak{r}(s))\|^{2}_{L_{Q}}ds\nonumber
\\&\quad+2\int^{\tau_{N}\wedge t}_{0}\langle{\bf u}^{\epsilon}(s), \sigma(\mathfrak{r}(s))dW(s)\rangle\nonumber
\\&\quad+\int^{\tau_{N}\wedge t}_{0}\int_{Z}\Big(|{\bf u}^{\epsilon}(s)+{\bf G}(\mathfrak{r}(s-), z)|^{2}-|{\bf u}^{\epsilon}(s)|^{2}\Big)\tilde{N}_{1}(dz, ds)\nonumber
\\&\quad+\int^{\tau_{N}\wedge t}_{0}\int_{Z}\Big(|{\bf u}^{\epsilon}(s)+{\bf G}(\mathfrak{r}(s-), z)|^{2}-|{\bf u}^{\epsilon}(s)|^{2}-2\big({\bf u}^{\epsilon}(s), {\bf G}(\mathfrak{r}(s-), z)\big)_{H}\Big)\nu_{1}(dz)ds.\nonumber
\end{align}
Taking expectation on the both side and using the basic Young inequality, we obtain
\begin{align*}
&\mathbb{E}|{\bf u}^{\epsilon}(\tau_{N}\wedge t)|^{2}+\nu\mathbb{E}\int^{\tau_{N}\wedge t}_{0}\|{\bf u}^{\epsilon}(s)\|^{2}ds
\\&\leq\mathbb{E}|{\bf u}^{\epsilon}(0)|^{2}+\frac{1}{\nu}\mathbb{E}\int^{\tau_{N}\wedge t}_{0}\|{\bf f}(s)\|^{2}_{V'}ds+\mathbb{E}\int^{\tau_{N}\wedge t}_{0}\|\sigma(\mathfrak{r}(s))\|^{2}_{L_{Q}}ds
\\&\quad+\mathbb{E}\int^{\tau_{N}\wedge t}_{0}\int_{Z}|{\bf G}(\mathfrak{r}(s-), z)|^{2}\nu_{1}(dz)ds.
\end{align*}
Utilizing Hypotheses $\bf A$ and (3) of Lemma \ref{convolution}, we simplify the above to obtain
\begin{align}\label{Tem L^{2} global in time MS paper}
\mathbb{E}|{\bf u}^{\epsilon}(\tau_{N}\wedge t)|^{2}+\nu\mathbb{E}\int^{\tau_{N}\wedge t}_{0}\|{\bf u}^{\epsilon}(s)\|^{2}ds
\leq\mathbb{E}|{\bf u}_{0}|^{2}+\frac{1}{\nu}\mathbb{E}\int^{t}_{0}\|{\bf f}(s)\|_{V'}^{2}ds+2Kt.
\end{align}

A simplification \eqref{Ito MS MS paper} gives
\begin{align}
&|{\bf u}^{\epsilon}(\tau_{N}\wedge t)|^{2}+2\nu\int^{\tau_{N}\wedge t}_{0}\|{\bf u}^{\epsilon}(s)\|^{2}ds\label{Ito MS' MS paper}
\\&=|{\bf u}^{\epsilon}(0)|^{2}+2\int^{\tau_{N}\wedge t}_{0}\langle{\bf f}(s), {\bf u}^{\epsilon}(s)\rangle ds+\int^{\tau_{N}\wedge t}_{0}\|\sigma(\mathfrak{r}(s))\|^{2}_{L_{Q}}ds\nonumber
\\&\quad+2\int^{\tau_{N}\wedge t}_{0}\langle{\bf u}^{\epsilon}(s), \sigma(\mathfrak{r}(s))dW(s)\rangle
+2\int^{\tau_{N}\wedge t}_{0}\int_{Z}\big({\bf u}^{\epsilon}(s), {\bf G}(\mathfrak{r}(s-), z)\big)\tilde{N}_{1}(dz, ds)\nonumber
\\&\quad+\int^{\tau_{N}\wedge t}_{0}\int_{Z}|{\bf G}(\mathfrak{r}(s-), z)|^{2}N_{1}(dz, ds)\nonumber.
\end{align}
For the last term in \eqref{Ito MS' MS paper}, by Hypothesis $\bf A2$, we have
\begin{align}
&\mathbb{E}\sup_{0\leq v\leq \tau_{N}\wedge t}\int^{v}_{0}\int_{Z}|{\bf G}(\mathfrak{r}(s-), z)|^{2}N_{1}(dz, ds)\label{Poisson MS paper}
=\mathbb{E}\int^{\tau_{N}\wedge t}_{0}\int_{Z}|{\bf G}(\mathfrak{r}(s-), z)|^{2}N_{1}(dz, ds)
\\&=\mathbb{E}\int^{\tau_{N}\wedge t}_{0}\int_{Z}|{\bf G}(\mathfrak{r}(s-), z)|^{2}\nu(dz)ds\leq Kt.\nonumber
\end{align}

One employs the Davis inequality, the basic Young inequality, and Hypothesis $\bf A1$ to obtain
\begin{align}
&2\mathbb{E}\sup_{0\leq v<\tau_{N}\wedge T}\Big|\int^{v}_{0}\langle{\bf u}^{\epsilon}(s), \sigma(\mathfrak{r}(s))dW(s)\rangle\Big|\label{martingale BM MS paper}
\\&\leq2\sqrt{2}\mathbb{E}\Big\{\Big(\int^{\tau_{N}\wedge T}_{0}\|\sigma^{\ast}(\mathfrak{r}(s)){\bf u}^{\epsilon}(s)\|^{2}_{L_{Q}}ds\Big)^{\frac{1}{2}}\Big\}
\leq2\sqrt{2}\epsilon_{1}\mathbb{E}\sup_{0\leq s\leq\tau_{N}\wedge T}|{\bf u}^{\epsilon}(s)|^{2}+2\sqrt{2}C_{\epsilon_{1}}KT,\nonumber
\end{align}
where $\epsilon_{1}>0$ will be chosen later. 

In a similar fashion, the Davis inequality, the basic Young inequality, and  Hypothesis $\bf A2$ imply
\begin{align}
&2\mathbb{E}\sup_{0\leq v\leq\tau_{N}\wedge T}\Big|\int^{v}_{0}\int_{Z}\big({\bf u}^{\epsilon}(s), {\bf G}(\mathfrak{r}(s-), z)\big)_{H}\tilde{N}_{1}(dz, ds)\Big|\label{martingale Poisson MS paper}
\\&\leq 2\sqrt{10}\mathbb{E}\Big\{\Big(\int^{\tau_{N}\wedge T}_{0}\int_{Z}\Big|\big({\bf u}^{\epsilon}(s), {\bf G}(\mathfrak{r}(s-), z)\big)_{H}\Big|^{2}\nu_{1}(dz)ds\Big)^{\frac{1}{2}}\Big\}\nonumber
\\&\leq 2\sqrt{10}\mathbb{E}\Big\{\Big(\int^{\tau_{N}\wedge T}_{0}\int_{Z}|{\bf u}^{\epsilon}(s)|^{2}|{\bf G}(\mathfrak{r}(s-), z)|^{2}\nu_{1}(dz)ds\Big)^{\frac{1}{2}}\Big\}\nonumber
\\&\leq 2\sqrt{10}\epsilon_{2}\mathbb{E}\sup_{0\leq s\leq \tau_{N}\wedge T}|{\bf u}^{\epsilon}(s)|^{2}+2\sqrt{10}C_{\epsilon_{2}}KT\nonumber,
\end{align}
where $\epsilon_{2}>0$ will be chosen later.

It follows from the basic Young inequality that
\begin{align}
2\mathbb{E}\int^{\tau_{N}\wedge T}_{0}\|{\bf f}(s)\|_{V'}\|{\bf u}^{\epsilon}(s)\|ds\leq\frac{3\nu}{2}\mathbb{E}\int^{\tau_{N}\wedge T}_{0}\|{\bf u}^{\epsilon}(s)\|^{2}ds+\frac{2}{3\nu}\mathbb{E}\int^{T}_{0}\|{\bf f}(s)\|^{2}_{V'}ds.\label{special Young's MS paper}
\end{align}

Now, taking supremum over $[0, \tau_{N}\wedge T]$ and then expectation on the both side of $\eqref{Ito MS' MS paper}$, using estimates \eqref{Poisson MS paper}, \eqref{martingale BM MS paper}, \eqref{martingale Poisson MS paper}, and \eqref{special Young's MS paper} in \eqref{Ito MS' MS paper}, we have
\begin{align*}
&\mathbb{E}\sup_{0\leq s\leq\tau_{N}\wedge T}|{\bf u}^{\epsilon}(s)|^{2}+\frac{1}{2}\nu\mathbb{E}\int^{\tau_{N}\wedge T}_{0}\|{\bf u}^{\epsilon}(s)\|^{2}ds
\\&\leq\mathbb{E}|{\bf u}^{\epsilon}(0)|^{2}+\frac{2}{3\nu}\mathbb{E}\int^{T}_{0}\|{\bf f}(s)\|^{2}_{V'}ds+KT
+2\sqrt{2}\epsilon_{1}\mathbb{E}\sup_{0\leq\tau_{N}\wedge T}|{\bf u}^{\epsilon}(s)|^{2}+2\sqrt{2}C_{\epsilon_{1}}KT
\\&\quad+2\sqrt{10}\epsilon_{2}\mathbb{E}\sup_{0\leq\tau_{N}\wedge T}|{\bf u}^{\epsilon}(s)|^{2}+2\sqrt{10}C_{\epsilon_{2}}KT+KT.
\end{align*}
Choosing $\epsilon_{1}=\frac{1}{8\sqrt{2}}$ and $\epsilon_{2}=\frac{1}{8\sqrt{10}}$, we have
\begin{align}
&\frac{1}{2}\mathbb{E}\sup_{0\leq s\leq \tau_{N}\wedge T}|{\bf u}^{\epsilon}(s)|^{2}+\frac{1}{2}\nu\mathbb{E}\int^{\tau_{N}\wedge T}_{0}\|{\bf u}^{\epsilon}(s)\|^{2}ds\label{Tem sup L^{2} global in time MS paper}
\\&\leq\mathbb{E}|{\bf u}_{0}|^{2}+\frac{2}{3\nu}\mathbb{E}\int^{T}_{0}\|{\bf f}(s)\|^{2}_{V'}ds+50KT,\nonumber
\end{align}
which implies that $\tau_{N}\rightarrow\infty$ as $N\rightarrow\infty$ almost surely.

Taking $N\rightarrow\infty$ in \eqref{Tem L^{2} global in time MS paper} and \eqref{Tem sup L^{2} global in time MS paper}, we obtain \eqref{L^2 strong MS paper} and  \eqref{L^2 sup strong MS paper}, respectively.
\end{proof}

\begin{remark*}
We remind the reader that both ${\bf u}^{\epsilon}(t)$ and ${\bf u}(t)$ exist on $[0, T]$ for arbitrary large $T$ by Proposition \ref{enhanced estimates}. Therefore, it is reasonable to discuss their asymptotic behaviors. 
\end{remark*}

\subsection{Exponential Stability}\label{exp sta MS paper}
We need the following two technical lemmata for establishing exponential stability.
Let
\begin{align}
M_{1}(t)&=\int^{t}_{0}\langle{\bf u}^{\epsilon}(s), \sigma(\mathfrak{r}(s))dW(s)\rangle\label{Mg_{1} MS paper}\\
M_{2}(t)&=\int^{t}_{0}\int_{Z}\Big(|{\bf u}^{\epsilon}(s)+{\bf G}(\mathfrak{r}(s-), z)|^{2}-|{\bf u}^{\epsilon}(s)|^{2}\Big)\tilde{N}_{1}(dz, ds)\label{Mg_{2} MS paper}.
\end{align}
Let us denote $M^{\ast}(T):=\sup_{t\in[0, T]}|M(t)|$ for a martingale $M(t)$.

\begin{lemma}\label{lemma M MS paper}
Assume that $\mathbb{E}(|{\bf u}_{0}|^{3})<\infty$ and ${\bf f}\in L^{3}(0, T; V')$. In addition, if 
\begin{align*}
\lim_{T\rightarrow\infty}\frac{1}{T}\int^{T}_{0}\|{\bf f}(s)\|^{2}_{V'}ds=F>0, 
\end{align*}
then there exist a sequence  $\{T_{n}\}$ with $T_{n}\rightarrow\infty$ as $n\rightarrow\infty$ such that 
$
\lim_{n\rightarrow\infty}M^{\ast}_{i}(T_{n})/T_{n}=0
$
almost surely for $i=1, 2$.
\end{lemma}

\begin{proof}
For $M_{1}(t)$, utilizing the Davis inequality, Hypothesis $\bf A1$, the property $|\cdot|\leq\|\cdot\|$, and the Schwarz inequality, we have
\begin{align*}
\mathbb{E}M^{\ast}_{1}(T)&=\mathbb{E}\sup_{t\in [0, T]}\Big|\int^{t}_{0}\langle{\bf u}^{\epsilon}(s), \sigma(\mathfrak{r}(s))dW(s)\rangle\Big|
\\&\leq\sqrt{2}\mathbb{E}\Big\{\Big(\int^{T}_{0}\|\sigma^{\ast}(\mathfrak{r}(s)){\bf u}^{\epsilon}(s)\|^{2}_{L_{Q}}ds\Big)^{\frac{1}{2}}\Big\}
\leq\sqrt{2}\mathbb{E}\Big\{\Big(\int^{T}_{0}\|\sigma^{\ast}(\mathfrak{r}(t))\|^{2}_{L_{Q}}|{\bf u}^{\epsilon}(s)|^{2}ds\Big)^{\frac{1}{2}}\Big\}
\\&\leq\sqrt{2K}\mathbb{E}\Big\{\Big(\int^{T}_{0}\|{\bf u}^{\epsilon}(s)\|^{2}ds\Big)^{\frac{1}{2}}\Big\}
\leq\sqrt{2K}\Big(\mathbb{E}\int^{T}_{0}\|{\bf u}^{\epsilon}(s)\|^{2}ds\Big)^{\frac{1}{2}}.
\end{align*}
In addition, it follows from  \eqref{L^2 sup strong MS paper} that
\begin{align*}
\Big(\mathbb{E}\int^{T}_{0}\|{\bf u}^{\epsilon}(s)\|^{2}ds\Big)^{\frac{1}{2}}\leq\Big(\frac{2}{\nu}\mathbb{E}|{\bf u}^{\epsilon}(0)|^{2}+\frac{4}{3\nu^{2}}\mathbb{E}\int^{T}_{0}\|{\bf f}(s)\|^{2}_{V'}ds+\frac{100KT}{\nu}\Big)^{\frac{1}{2}}.
\end{align*}
Therefore, by (3) of Lemma \ref{convolution}, we have
\begin{align*}
\frac{\mathbb{E}M^{\ast}_{1}(T)}{T}&\leq\frac{\sqrt{2K}}{T}\Big(\frac{2}{\nu}\mathbb{E}|{\bf u}^{\epsilon}(0)|^{2}+\frac{4}{3\nu^{2}}\mathbb{E}\int^{T}_{0}\|{\bf f}(s)\|^{2}_{V'}ds+\frac{100KT}{\nu}\Big)^{\frac{1}{2}}
\\&\leq\frac{\sqrt{2K}}{\sqrt{T}}\Big(\frac{2}{\nu}\frac{\mathbb{E}|{\bf u}_{0}|^{2}}{T}+\frac{4}{3\nu^{2}}\frac{\mathbb{E}\int^{T}_{0}\|{\bf f}(s)\|^{2}_{V'}ds}{T}+\frac{100K}{\nu }\Big)^{\frac{1}{2}}.
\end{align*}
Moreover, it follows from the assumption of ${\bf f}$ that
$
\mathbb{E}M^{\ast}_{1}(T)/T\rightarrow 0
$
as $T\rightarrow\infty$. Thus, there exists a subsequence $\{T_{1, n}\}$ such that
$
M^{\ast}_{1}(T_{1, n})/T_{1, n}\rightarrow 0,
$
as $n\rightarrow\infty$, almost surely.

For $M_{2}(t)$, we have
\begin{align*}
M_{2}(t)=\int^{t}_{0}\int_{Z}2\big({\bf u}^{\epsilon}(s), {\bf G}(\mathfrak{r}(s-), z)\big)\tilde{N}_{1}(dz, ds)+\int^{t}_{0}\int_{Z}|{\bf G}(\mathfrak{r}(s-), z)|^{2}\tilde{N}_{1}(dz, ds),
\end{align*}
therefore, the Davis inequality, Hypothesis $\bf A2$, the property $|\cdot|\leq\|\cdot\|$, the Schwarz inequality, and \eqref{L^2 sup strong MS paper} imply
\begin{align*}
\mathbb{E}M^{\ast}_{2}(T)&\leq \sqrt{10}\mathbb{E}\Big\{\Big(\int^{T}_{0}\int_{Z}\Big|\big({\bf u}^{\epsilon}(s), {\bf G}(\mathfrak{r}(s-), z)\big)\Big|^{2}\nu_{1}(dz)ds\Big)^{\frac{1}{2}}\Big\}
\\&\quad+\sqrt{10}\mathbb{E}\Big\{\Big(\int^{T}_{0}\int_{Z}|{\bf G}(\mathfrak{r}(s-), z)|^{4}\nu_{1}(dz)ds\Big)^{\frac{1}{2}}\Big\}
\\&\leq \sqrt{10}\mathbb{E}\Big\{\Big(\int^{T}_{0}\int_{Z}|{\bf u}^{\epsilon}(s)|^{2}|{\bf G}(\mathfrak{r}(s-), z)|^{2}ds\Big)^{\frac{1}{2}}\Big\}+\sqrt{10KT}
\\&\leq \sqrt{10K}\mathbb{E}\Big\{\Big(\int^{T}_{0}\|{\bf u}^{\epsilon}(s)\|^{2}ds\Big)^{\frac{1}{2}}\Big\}+\sqrt{10KT}
\\&\leq \sqrt{10K}\Big(\mathbb{E}\int^{T}_{0}\|{\bf u}^{\epsilon}(s)\|^{2}ds\Big)^{\frac{1}{2}}+\sqrt{10KT}
\\&\leq \sqrt{10K}\Big(\frac{2}{\nu}\mathbb{E}|{\bf u}^{\epsilon}(0)|^{2}+\frac{4}{3\nu^{2}}\mathbb{E}\int^{T}_{0}\|{\bf f}(s)\|^{2}_{V'}ds+\frac{100KT}{\nu}\Big)^{\frac{1}{2}}
+\sqrt{10KT}.
\\&\leq \sqrt{10K}\Big(\frac{2}{\nu}\mathbb{E}|{\bf u}_{0}|^{2}+\frac{4}{3\nu^{2}}\mathbb{E}\int^{T}_{0}\|{\bf f}(s)\|^{2}_{V'}ds+\frac{100KT}{\nu}\Big)^{\frac{1}{2}}
+\sqrt{10KT},
\end{align*}
where the last inequality follows from (3) in Lemma \ref{convolution}. This implies that
$
\mathbb{E}M^{\ast}_{2}(T)/T\rightarrow 0
$
as $T\rightarrow\infty$. Therefore, by an analogous argument as for $M_{1}(t)$, there exist a sequence $\{T_{2, n}\}$ such that
$
M^{\ast}_{2}(T_{2, n})/T_{2, n}\rightarrow 0
$
almost surely as $n\rightarrow\infty$.

Let $\{T_{n}\}$ be a common subsequence of $\{T_{1, n}\}$ and $\{T_{2, n}\}$. Then 
\begin{align*}
\lim_{n\rightarrow\infty}\frac{M^{\ast}_{i}(T_{n})}{T_{n}}=0
\end{align*}
almost surely for $i=1, 2$.
\end{proof}

\begin{remark*}
The argument employed in Lemma \ref{lemma M MS paper} is in the context of stochastic Navier-Stokes equations. One may employ other methods to deduce such a limit. For instance, one may utilize \cite[Lemma. 2.1]{Sundar LIL} to obtain the almost surely limits of $M_{1}$ and $M_{2}$ of the original sequence instead of a subsequence (cf. \cite[Eq. (3.17)]{Sundar measure}).
\end{remark*}

Recall that $\lambda_{1}$ is the first eigenvalue of the Stokes operator ${\bf A}$, and $K$ is the constant in the Hypotheses $\bf A$.
\begin{lemma}\label{lemma v norm of u MS paper}
Let $\epsilon>0$ be fixed. Assume that $\mathbb{E}(|{\bf u}_{0}|^{3})<\infty$ and ${\bf f}\in L^{3}(0, T; V')$. In addition, if 
\begin{align*}
\lim_{T\rightarrow\infty}\frac{1}{T}\int^{T}_{0}\|{\bf f}(s)\|^{2}_{V'}ds=F>0 
\quad\mbox{and}\quad
K<\frac{\nu^{3}\lambda_{1}-F/\nu}{2},
\end{align*}
then 
\begin{align*}
\lim_{T\rightarrow\infty}\Big(\nu\lambda^{}_{1}-\frac{1}{\nu T}\int^{T}_{0}\|{\bf u}^{\epsilon}(s)\|^{2}ds\Big)>0
\end{align*}
almost surely.
\end{lemma}

\begin{proof}
It follows from the It\^o formula, the basic Young inequality, and (3) in Lemma \ref{convolution} that
\begin{align*}
&\sup_{t\in[0, T]}|{\bf u}^{\epsilon}(t)|^{2}+\nu\int^{T}_{0}\|{\bf u}^{\epsilon}(s)\|^{2}ds
\\&\leq|{\bf u}_{0}|^{2}+\frac{1}{\nu}\int^{T}_{0}\|{\bf f}(s)\|^{2}_{V'}ds+\int^{T}_{0}\|\sigma(\mathfrak{r}(s))\|^{2}_{L_{Q}}ds+2M^{\ast}_{1}(T)+M^{\ast}_{2}(T)
\\&\quad+\int^{T}_{0}\int_{Z}|{\bf G}(\mathfrak{r}(s-), z)|^{2}\nu(dz)ds,
\end{align*}
where $M_{1}(T)$ and $M_{2}(T)$ are defined as in \eqref{Mg_{1} MS paper} and \eqref{Mg_{2} MS paper}, respectively, and $M^{\ast}_{i}(T)$, $i=1, 2$, are introduced before Lemma \ref{lemma M MS paper}. Using Hypotheses $\bf A$ and dropping $\sup_{t\in [0, T]}|{\bf u}^{\epsilon}(t)|^{2}$, we have
\begin{align*}
&\frac{1}{\nu T}\int^{T}_{0}\|{\bf u}^{\epsilon}(s)\|^{2}ds
\\&\leq\frac{|{\bf u}_{0}|^{2}}{\nu^{2}T}+\frac{1}{\nu^{3}}\frac{\int^{T}_{0}\|{\bf f}(s)\|^{2}_{V'}ds}{T}+\frac{2KT}{\nu^{2}T}
+\frac{2M^{\ast}_{1}(T)}{\nu^{2}T}+\frac{M^{\ast}_{2}(T)}{\nu^{2}T}, 
\end{align*}
which implies
\begin{align*}
&\nu\lambda^{\epsilon}_{1}-\frac{1}{\nu T}\int^{T}_{0}\|{\bf u}^{\epsilon}(s)\|^{2}ds
\\&\geq\nu\lambda^{\epsilon}_{1}-\Big(\frac{|{\bf u}_{0}|^{2}}{\nu^{2}T}+\frac{1}{\nu^{3}}\frac{\int^{T}_{0}\|{\bf f}(s)\|^{2}_{V'}ds}{T}+\frac{2K}{\nu^{2}}
+\frac{2M^{\ast}_{1}(T)}{\nu^{2}T}+\frac{M^{\ast}_{2}(T)}{\nu^{2}T}\Big).
\end{align*}
By Lemma \ref{lemma M MS paper}, the assumption of ${\bf f}$, and the requirement of $K$, we conclude
\begin{align*}
\lim_{n\rightarrow\infty}\Big(\nu\lambda_{1}-\frac{1}{\nu T_{n}}\int^{T_{n}}_{0}\|{\bf u}^{\epsilon}(s)\|^{2}ds\Big)\geq\nu\lambda_{1}-\frac{F}{\nu^{3}}-\frac{2K}{\nu^{2}}>0
\end{align*}
almost surely.
\end{proof}

Let ${\bf u}^{\epsilon}_{i}(t)$ be the solution of \eqref{equation 1 MS paper} with initial conditions ${\bf u}^{\epsilon}_{i}(0)={\bf u}^{\epsilon}_{i}=k_{\epsilon}{\bf u}_{i}$ and $\mathfrak{r}_{i}(0)=\mathfrak{r}_{i}$. Write ${\bf w}^{\epsilon}(t)={\bf u}^{\epsilon}_{1}(t)-{\bf u}^{\epsilon}_{2}(t)$, $\sigma_{12}(t)=\sigma(\mathfrak{r}_{1}(t))-\sigma(\mathfrak{r}_{2}(t))$, and ${\bf G}_{12}(t-, z)={\bf G}(\mathfrak{r}_{1}(t-), z)-{\bf G}(\mathfrak{r}_{2}(t-), z)$.

Now we are in the position to prove the exponential stability.

\begin{theorem}[Exponential stability]\label{exponential stability MS paper}
Let $\epsilon>0$ be fixed. Assume the hypotheses in Lemma \ref{lemma v norm of u MS paper} is valid. 
Then we have
\begin{align*}
\lim_{t\rightarrow\infty}|{\bf u}^{\epsilon}_{1}(t)-{\bf u}^{\epsilon}_{2}(t)|^{2}=0
\end{align*}
for almost all $\omega\in\Omega$.
\end{theorem}

\begin{proof}
It follows from the It\^o formula that
\begin{align}
&|{\bf w}^{\epsilon}(t)|^{2}+2\nu\int^{t}_{0}\|{\bf w}^{\epsilon}(s)\|^{2}ds+\int^{t}_{0}\langle{\bf B}_{k_{\epsilon}}({\bf u}^{\epsilon}_{1}(s))-{\bf B}_{k_{\epsilon}}({\bf u}^{\epsilon}_{2}(s)), {\bf u}^{\epsilon}_{1}(s)-{\bf u}^{\epsilon}_{2}(s)\rangle ds\label{Ito of w_exponential stability MS paper}
\\&=|{\bf w}^{\epsilon}_{0}|^{2}+\int^{t}_{0}\|\sigma_{12}(s)\|^{2}_{L_{Q}}ds+2\int^{t}_{0}\langle{\bf w}^{\epsilon}(s), \sigma_{12}(s)dW(s)\rangle\nonumber
\\&\quad+\int^{t}_{0}\int_{Z}\Big(|{\bf w}^{\epsilon}(s)+{\bf G}_{12}(s-, z)|^{2}-|{\bf w}^{\epsilon}(s)|^{2}\Big)\tilde{N}_{1}(dz, ds)\nonumber
\\&\quad+\int^{t}_{0}\int_{Z}\Big(|{\bf w}^{\epsilon}(s)+{\bf G}_{12}(s-, z)|^{2}-|{\bf w}^{\epsilon}(s)|^{2}-2({\bf w}^{\epsilon}(s), {\bf G}_{12}(s-, z))\Big)\nu(dz)ds.\nonumber
\end{align}
For the nonlinear term, by using \eqref{b_{k} uvu}, we have
\begin{align*}
|\langle{\bf B}_{k_{\epsilon}}({\bf u}^{\epsilon}_{1}(s))-{\bf B}_{k_{\epsilon}}({\bf u}^{\epsilon}_{2}(s)), {\bf u}^{\epsilon}_{1}(s)-{\bf u}^{\epsilon}_{2}(s)\rangle|\leq\|{\bf w}^{\epsilon}(s)\||{\bf w}^{\epsilon}(s)|\|{\bf u}^{\epsilon}_{1}(s)\|,
\end{align*}
therefore, one infers from the basic Young inequality that
\begin{align*}
&2\int^{t}_{0}|\langle{\bf B}_{k_{\epsilon}}({\bf u}^{\epsilon}_{1}(s))-{\bf B}_{k_{\epsilon}}({\bf u}^{\epsilon}_{2}(s)), {\bf u}^{\epsilon}_{1}(s)-{\bf u}^{\epsilon}_{2}(s)\rangle|ds
\\&\leq\nu\int^{t}_{0}\|{\bf w}^{\epsilon}(s)\|^{2}ds+\frac{1}{\nu}\int^{t}_{0}|{\bf w}^{\epsilon}(s)|^{2}\|{\bf u}^{\epsilon}_{1}(s)\|^{2}ds.
\end{align*}

Let
\begin{align*}
\tilde{M}_{1}(t)&=\int^{t}_{0}\langle{\bf w}^{\epsilon}(s), \sigma_{12}(s)dW(s)\rangle\\
\tilde{M}_{2}(t)&=\int^{t}_{0}\int_{Z}\Big(|{\bf w}^{\epsilon}(s)+{\bf G}_{12}(s-, z)|^{2}-|{\bf w}^{\epsilon}(s)|^{2}\Big)\tilde{N}_{1}(dz, ds).
\end{align*}
Using the notation $\tilde{M}_{i}(t)$, $i=1, 2$, the estimate for nonlinear term, and Hypotheses $\bf A$ in \eqref{Ito of w_exponential stability MS paper}, we obtain
\begin{align*}
&|{\bf w}^{\epsilon}(t)|^{2}+\nu\int^{t}_{0}\|{\bf w}^{\epsilon}(s)\|^{2}ds
\\&\leq|{\bf w}^{\epsilon}_{0}|^{2}+\frac{1}{\nu}\int^{t}_{0}|{\bf w}^{\epsilon}(s)|^{2}\|{\bf u}^{\epsilon}_{1}(s)\|^{2}ds+2Kt+2\tilde{M}^{\ast}_{1}(T)+\tilde{M}^{\ast}_{2}(T)
+2Kt
\end{align*}

The Poincar\'e inequality \eqref{Poincare ineq} implies that
$
\lambda_{1}|{\bf w}^{\epsilon}(s)|^{2}\leq\|{\bf w}^{\epsilon}(s)\|^{2}.
$
Plugging this result into above and using (3) of Lemma \ref{convolution}, we have
\begin{align*}
&|{\bf w}^{\epsilon}(t)|^{2}+\int^{t}_{0}(\nu\lambda_{1}-\frac{1}{\nu}\|{\bf u}^{\epsilon}_{1}(s)\|^{2})|{\bf w}^{\epsilon}(s)|^{2}ds
\\&\leq |{\bf w}_{0}|^{2}+2\tilde{M}^{\ast}_{1}(T)+\tilde{M}^{\ast}_{2}(T)+4KT.
\end{align*}
Writing $C_{T}=|{\bf w}_{0}|^{2}+2\tilde{M}^{\ast}_{1}(T)+\tilde{M}^{\ast}_{2}(T)+4KT$, one infers from the Gronwall inequality that
$
|{\bf w}^{\epsilon}(t)|^{2}\leq C_{T}e^{-\int^{T}_{0}(\nu\lambda_{1}-\frac{1}{\nu}\|{\bf u}^{\epsilon}_{1}(s)\|^{2})ds}=C_{T}e^{-(\nu\lambda_{1}-\frac{1}{\nu T}\int^{T}_{0}\|{\bf u}^{\epsilon}(s)\|^{2}ds)T}.
$
The theorem follows from Lemma \ref{lemma v norm of u MS paper} and the fact that $C_{T}$ is growing as a polynomial in $T$.
\end{proof}

\subsection{Stationary Measure}\label{stationary measures MS paper}
Through out this subsection, the external forcing ${\bf f}$ is assumed to be independent of  time. Fix $x\in H$ and $i\in\mathcal{S}$ as the initial condition for ${\bf u}^{\epsilon}(t)$ and $\mathfrak{r}(t)$, respectively. For all $B\in\mathcal{B}(H)$ and $j\in\mathcal{S}$, let $\mathcal{Q}^{\epsilon}_{t}(x, i; B, j)$ denote the transition probability
 of $({\bf u}^{\epsilon}, \mathfrak{r})$:
\begin{align*}
\mathcal{Q}^{\epsilon}_{t}(x, i; B, j)=P(({\bf u}^{\epsilon}(t), \mathfrak{r}(t))\in (B, j)| ({\bf u}^{\epsilon}(0), \mathfrak{r}(0))=(x, i)).
\end{align*}

Let $\phi: H\times\mathcal{S}\rightarrow\mathbb{R}$ be a bounded continuous function. Define
\begin{align}
(\mathcal{Q}^{\epsilon}_{t}\phi)(x, i):&=\sum_{k=1}^{m}\int_{H}\phi(z, k)\mathcal{Q}^{\epsilon}_{t}(x, i; dz, k)\label{P_t MS paper}
=\mathbb{E}^{x, i}(\phi({\bf u}^{\epsilon}(t), \mathfrak{r}(t))),
\end{align}
where the upperscript $x, i$ of $\mathbb{E}$ is for emphasizing the initial condition of $({\bf u}(t), \mathfrak{r}(t))$ . In addition, 
\begin{align}\label{showing Q is stochastic conti}
\begin{split}
\lim_{t\rightarrow 0}(\mathcal{Q}^{\epsilon}_{t}\phi)(x, i)&=\lim_{t\rightarrow 0}\mathbb{E}^{x, i}(\phi({\bf u}^{\epsilon}(t), \mathfrak{r}(t)))=\mathbb{E}^{x, i}\lim_{t\rightarrow 0}\phi({\bf u}^{\epsilon}(t), \mathfrak{r}(t))
\\&=\mathbb{E}^{x, i}\phi({\bf u}^{\epsilon}(0), \mathfrak{r}(0))=\mathbb{E}\phi(x, i)=\phi(x, i),
\end{split}
\end{align}
by the bounded convergence theorem. Therefore, the transition probability $\mathcal{Q}^{\epsilon}_{t}$ is stochastic continuous (see, e.g., \cite[Proposition 2.1.1]{DZ}). 

We say $\lambda^{\epsilon}$ is a stationary measure if
\begin{align*}
\sum_{j=1}^{m}\int_{H}(\mathcal{Q}^{\epsilon}_{t}\phi)(y, j)\lambda^{\epsilon}(dy, j)=\sum_{j=1}^{m}\int_{H}\phi(y, j)\lambda^{\epsilon}(dy, j)
\end{align*}
for all  $t\geq 0$ and $\phi(\cdot, j)\in C_{b}(H)$ for all $j\in S$; if $\lambda^{\epsilon}$ is unique, then it is ergodic (see, \cite[Theorem 3.2.6]{DZ}).

Now we are in the position to construct the ergodic measure for the regularized system.
\begin{theorem}\label{existence of stationary MS paper}
Let $\epsilon>0$ be fixed. Assume that $\mathbb{E}(|{\bf u}_{0}|^{3})<\infty$ and $\|{\bf f}\|^{2}_{V'}=F>0$. If $K$ satisfies the assumption in Lemma \ref{lemma v norm of u MS paper}, then there exists an ergodic measure $\lambda^{\epsilon}$, with support in $V\times\mathcal{S}$ and a finite second moment, for the solution ${\bf u}^{\epsilon}$ of the equation \eqref{equation 1 MS paper} with initial condition $({\bf u}^{\epsilon}(0), \mathfrak{r}(0))=(x, i)$.
\end{theorem}

\begin{proof}
We begin the proof by showing the existence. It follows from \eqref{L^2 strong MS paper} that
\begin{align}
\frac{\nu}{t}\mathbb{E}\int^{t}_{0}\|{\bf u}^{\epsilon}(s)\|^{2}ds\leq C\label{4.5 MS paper}
\end{align}
for $t>1$ and $C$ is an appropriate constant independent of $t$ and $\epsilon$. Hence by the Chebyshev inequality
\begin{align}
\lim_{N\rightarrow\infty}\sup_{t>1}\frac{1}{t}\int^{t}_{0}P(\|{\bf u}^{\epsilon}(s)\|>N)ds=0\label{4.6 MS paper}
\end{align}
follows.

Let $\{t_{n}\}$ be any increasing sequence of positive numbers with $\lim_{n\rightarrow\infty}t_{n}=\infty$. Define probability measures $\lambda^{\epsilon}_{n}$ as follows:
\begin{align}
\lambda^{\epsilon}_{n}(B, j)= \frac{1}{t_{n}}\int^{t_{n}}_{0}\mathcal{Q}^{\epsilon}_{s}(x, i; B, j)ds\label{lambda n MS paper}
\end{align}
for all $B\in\mathcal{B}(H)$ and $j\in\mathcal{S}$.

Let $N$ be an positive integer and  $A:=\{v\in V : \|v\|\leq N\}$, which is a bounded set in $V$. Then by the compact embedding $\iota:V\hookrightarrow H$, $A$ is a relative compact set in $H$. Consider
\begin{align*}
\lambda^{\epsilon}_{n}(A^{c}, j)&=\frac{1}{t_{n}}\int^{t_{n}}_{0}\mathcal{Q}^{\epsilon}_{s}(x, i; A^{c}, j)ds
\\&=\frac{1}{t_{n}}\int^{t_{n}}_{0}P(({\bf u}^{\epsilon}(t), \mathfrak{r}(t))\in (A^{c}, j) | ({\bf u}^{\epsilon}(0), \mathfrak{r}(0))=(x, i))
\\&=\frac{1}{t_{n}}\int^{t_{n}}_{0}P(\|{\bf u}^{\epsilon}(t)\|>N \ \mbox{and} \ \mathfrak{r}(t)=j  | {\bf u}^{\epsilon}(0)=x \  \mbox{and}\ \mathfrak{r}(0)=i),
\end{align*}
which together with \eqref{4.6 MS paper} further imply
\begin{align}\label{lambda epsilon n}
\lambda^{\epsilon}_{n}(A^{c}, j)=\frac{1}{t_{n}}\int^{t_{n}}_{0}P(\|{\bf u}^{\epsilon}(t)\|>N \ \mbox{and} \ \mathfrak{r}(t)=j  | {\bf u}^{\epsilon}(0)=x \  \mbox{and}\ \mathfrak{r}(0)=i)\rightarrow 0
\end{align}
as $N\rightarrow\infty$. Hence, for each fixed $\epsilon>0$, $\{\lambda^{\epsilon}_{n}\}_{n\in\mathbb{N}}$ is tight in the space of probability measures on $(H\times\mathcal{S}, \mathcal{B}(H\times\mathcal{S}))$ equipped with weak topology. Therefore, Prokhorov's theorem implies that there exists a subsequence $\{\lambda^{\epsilon}_{{n_{\alpha}}}\}$ such that 
\begin{align}\label{lambda epsilon n to lambda epsilon}
\lambda^{\epsilon}_{n_{\alpha}}\quad\text{converges in law to a limit}\quad\lambda^{\epsilon}\quad\mbox{as}\quad\alpha\rightarrow\infty.
\end{align}

For $\phi(\cdot, j)\in C_{b}(H)$ for all $j\in \mathcal{S}$, we have
\begin{align*}
&\sum_{j=1}^{m}\int_{H}(\mathcal{Q}^{\epsilon}_{t}\phi)(y, j)\lambda^{\epsilon}(dy, j)
=\lim_{{\alpha}\rightarrow\infty}\sum_{j=1}^{m}\int_{H}(\mathcal{Q}^{\epsilon}_{t}\phi)(y, j)\lambda^{\epsilon}_{n_{\alpha}}(dy, j)
\\&=\lim_{\alpha\rightarrow\infty}\sum_{j=1}^{m}\int_{H}(\mathcal{Q}^{\epsilon}_{t}\phi)(y, j)\frac{1}{t_{n_{\alpha}}}\int^{t_{n_{\alpha}}}_{0}\mathcal{Q}^{\epsilon}_{s}(x, i; dy, j)ds.
\end{align*}
Using \eqref{P_t MS paper} in above, we further have
\begin{align}
&\sum_{j=1}^{m}\int_{H}(\mathcal{Q}^{\epsilon}_{t}\phi)(y, j)\lambda^{\epsilon}(dy, j)\label{stationary lambda MS paper}
\\&=\lim_{\alpha\rightarrow\infty}\sum_{j=1}^{m}\sum_{k=1}^{m}\int_{H}\int_{H}\phi(z, k)\mathcal{Q}^{\epsilon}_{t}(y, j; dz, k)\frac{1}{t_{n_{\alpha}}}\int^{t_{n_{\alpha}}}_{0}\mathcal{Q}^{\epsilon}_{s}(x, i; dy, j)ds.\nonumber
\end{align}

Recalling the Chapman-Kolmogorov equation, i.e.,
\begin{align*}
\mathcal{Q}^{\epsilon}_{s+t}(x, i; dz, k)=\sum_{j=1}^{m}\int_{H}\mathcal{Q}^{\epsilon}_{s}(x, i; dy, j)\mathcal{Q}^{\epsilon}_{t}(y, j; dz, k)
\end{align*}
and plugging it into \eqref{stationary lambda MS paper}, we have
\begin{align*}
\sum_{j=1}^{m}\int_{H}(\mathcal{Q}^{\epsilon}_{t}\phi)(y, j)\lambda^{\epsilon}(dy, j)=\lim_{\alpha\rightarrow\infty}\sum_{k=1}^{m}\int_{H}\phi(z, k)\frac{1}{t_{n_{\alpha}}}\int^{t_{n_{\alpha}}}_{0}\mathcal{Q}^{\epsilon}_{s+t}(x, i; dz, k)ds.
\end{align*}

Note that
\begin{align*}
&\lim_{\alpha\rightarrow\infty}\frac{1}{t_{n_{\alpha}}}\int^{t_{n_{\alpha}}}_{0}\mathcal{Q}^{\epsilon}_{s+t}(x, i; dz, k)ds
=\lim_{\alpha\rightarrow\infty}\frac{1}{t_{n_{\alpha}}}\int^{t_{n_{\alpha}}+t}_{t}\mathcal{Q}^{\epsilon}_{u}(x, i; dz, k)du
\\&=\lim_{\alpha\rightarrow\infty}\frac{1}{t_{n_{\alpha}}}\Big(\int^{t_{n_{\alpha}}}_{0}\mathcal{Q}^{\epsilon}_{u}(x, i; dz, k)du+\int^{t_{n_{\alpha}}+t}_{t_{n_{\alpha}}}\mathcal{Q}^{\epsilon}_{u}(x, i; dz, k)du
\\&\qquad\qquad\quad-\int^{t}_{0}\mathcal{Q}^{\epsilon}_{u}(x, i; dz, k)du\Big)
\\&=\lim_{\alpha\rightarrow\infty}\frac{1}{t_{n_{\alpha}}}\int^{t_{n_{\alpha}}}_{0}\mathcal{Q}^{\epsilon}_{u}(x, i; dz, k)du
\end{align*}

Thus, we have
\begin{align*}
&\sum_{j=1}^{m}\int_{H}(\mathcal{Q}^{\epsilon}_{t}\phi)(y, j)\lambda^{\epsilon}(dy, j)
=\lim_{\alpha\rightarrow\infty}\sum_{k=1}^{m}\int_{H}\phi(z, k)\frac{1}{t_{n_{\alpha}}}\int^{t_{n_{\alpha}}}_{0}\mathcal{Q}^{\epsilon}_{s+t}(x, i; dz, k)ds
\\&=\lim_{\alpha\rightarrow\infty}\sum_{k=1}^{m}\int_{H}\phi(z, k)\frac{1}{t_{n_{\alpha}}}\int^{t_{n_{\alpha}}}_{0}\mathcal{Q}^{\epsilon}_{u}(x, i; dz, k)du
=\sum_{k=1}^{m}\int_{H}\phi(z, k)\lambda^{\epsilon}(dz, k).
\end{align*}

Hence, we conclude that $\lambda^{\epsilon}$ is a stationary measure.

For the second moment of $\lambda^{\epsilon}$, one employs the lower semi-continuity of the $H$-norm to deduce
\begin{align*}
\sum_{j=1}^{m}\int_{H}(|y|^{2}+j^{2})\lambda^{\epsilon}(dy, j)\leq\liminf_{\alpha\rightarrow\infty}\sum_{j=1}^{m}\int_{H}(|y|^{2}+j^{2})\lambda^{\epsilon}_{n_{\alpha}}(dy, j).
\end{align*}
Then using \eqref{P_t MS paper} and \eqref{lambda n MS paper} with $\phi(y, j)=|y|^{2}+j^{2}$, one further deduce
\begin{align*}
&\sum_{j=1}^{m}\int_{H}(|y|^{2}+j^{2})\lambda^{\epsilon}_{n_{\alpha}}(dx, i)
=\sum_{j=1}^{m}\int_{H}\phi(y, j)\frac{1}{n_{\alpha}}\int^{n_{\alpha}}_{0}\mathcal{Q}^{\epsilon}_{t}(x, i; dy, j)dt
\\&=\frac{1}{n_{\alpha}}\int^{n_{\alpha}}_{0}\sum_{j=1}^{m}\int_{H}\phi(y, j)\mathcal{Q}^{\epsilon}_{t}(x, i; dy, j)dt
=\frac{1}{n_{\alpha}}\int^{n_{\alpha}}_{0}\mathbb{E}(|{\bf u}^{\epsilon, y}(t)|^{2}+i^{2})dt<\infty
\end{align*}
by Proposition \ref{enhanced estimates}. Therefore, the second moment for the measure $\lambda^{\epsilon}$ exists.

For the support of $\lambda^{\epsilon}$, let ${\bf u}^{\epsilon}$ denote the solution of \eqref{equation 1 MS paper} started at ${\bf u}^{\epsilon}_{0}$ and ${\mathfrak{r}}_{0}$, where the (joint) distribution of $({\bf u}^{\epsilon}_{0}, \mathfrak{r}_{0})$ is given by $\lambda^{\epsilon}$. By \eqref{4.5 MS paper}, it follows that ${\bf u}^{\epsilon}(s)$ is almost surely $V$-valued for almost all $s$. In particular, $\lambda^{\epsilon}$ has support in $V\times\mathcal{S}$.

Let $\lambda^{\epsilon}_{1}$ and $\lambda^{\epsilon}_{2}$ be two stationary measures. To show the uniqueness, it suffices to show that
\begin{align*}
\sum_{i=1}^{m}\int_{H}\phi(x, i)\lambda^{\epsilon}_{1}(dx, i)=\sum_{i=1}^{m}\int_{H}\phi(x, i)\lambda^{\epsilon}_{2}(dx, i)
\end{align*}
for all $\phi(\cdot, i)\in C_{b}(H)$ for all $i\in \mathcal{S}$.

Let ${\bf u}^{\epsilon, x}$ denote the solution of \eqref{equation 1 MS paper} with ${\bf u}^{\epsilon}(0)=x$ and $\mathfrak{r}^{i}(t)$ the Markov chain $\mathfrak{r}(t)$  with $\mathfrak{r}(0)=i$. Define
\begin{align*}
\lambda^{\epsilon, (x, i)}_{T}(B, j)=\frac{1}{T}\int^{T}_{0}\mathcal{Q}^{\epsilon}_{t}(x, i; B, j)dt
\end{align*}
for all $B\in\mathcal{B}(H)$ and $j\in\mathcal{S}$. Let $\lambda^{\epsilon}(dx, i)$ denote a stationary measure. Then by stationarity,
\begin{align*}
&\sum_{i=1}^{m}\int_{H}\phi(x, i)\lambda^{\epsilon}(dx, i)
=\sum_{i=1}^{m}\sum_{j=1}^{m}\int_{H}\int_{H}\phi(y, j)\lambda^{\epsilon, (x, i)}_{T}(dy, j)\lambda^{\epsilon}(dx, i)
\\&=\sum_{i=1}^{m}\sum_{j=1}^{m}\int_{H}\int_{H}\phi(y, j)\lambda^{\epsilon, (x, i)}_{T}(dy, j)\lambda^{\epsilon}(dx | i)\pi_{i},
\end{align*} 
where $\pi=(\pi_{1}, \cdots,\pi_{m})$ is the unique stationary distribution of the Markov chain $\mathfrak{r}(t)$.
Furthermore, one obtains
\begin{align*}
&\sum_{i=1}^{m}\int_{H}\phi(x, i)\lambda^{\epsilon}(dx, i)
=\sum_{i=1}^{m}\sum_{j=1}^{m}\int_{H}\int_{H}\phi(y, j)\lambda^{\epsilon, (x, i)}_{T}(dy, j)\lambda^{\epsilon}(dx | i)\pi_{i}
\\&=\sum_{i=1}^{m}\int_{H}\frac{1}{T}\int^{T}_{0}\mathbb{E}^{x, i}(\phi({\bf u}^{\epsilon}(t), \mathfrak{r}(t)))dt\lambda^{\epsilon}(dx | i)\pi_{i}
\\&=\sum_{i=1}^{m}\int_{H}\frac{1}{T}\int^{T}_{0}\mathbb{E}(\phi({\bf u}^{\epsilon, x}(t), \mathfrak{r}^{i}(t)))dt\lambda^{\epsilon}(dx | i)\pi_{i}.
\end{align*}

Hence,
\begin{align}
&\Big|\sum_{i=1}^{m}\int_{H}\phi(x, i)\lambda^{\epsilon}_{1}(dx, i)-\sum_{i=1}^{m}\int_{H}\phi(w, i)\lambda^{\epsilon}_{2}(dw, i)\Big|\nonumber
\\&=\Big|\sum_{i=1}^{m}\int_{H}\frac{1}{T}\int^{T}_{0}\mathbb{E}(\phi({\bf u}^{\epsilon, x}(t), \mathfrak{r}^{i}(t)))dt\lambda^{\epsilon}_{1}(dx | i)\pi_{i}\nonumber
\\&\qquad-\sum_{i=1}^{m}\int_{H}\frac{1}{T}\int^{T}_{0}\mathbb{E}(\phi({\bf u}^{\epsilon, w}(t), \mathfrak{r}^{i}(t)))dt\lambda^{\epsilon}_{2}(dw | i)\pi_{i}\Big|\nonumber
\\&\leq\sum_{i=1}^{m}\int_{H}\int_{H}\frac{1}{T}\int^{T}_{0}\mathbb{E}\Big|\phi({\bf u}^{\epsilon, x}(t), \mathfrak{r}^{i}(t))-\phi({\bf u}^{\epsilon, w}(t), \mathfrak{r}^{i}(t))\Big|dt\lambda^{\epsilon}_{1}(dx | i)\lambda^{\epsilon}_{2}(dw | i)\pi_{i}.\label{4.9 MS paper}
\end{align}
By Theorem \ref{exponential stability MS paper} and the continuity of $\phi$, we have
\begin{align*}
\Big|\phi({\bf u}^{\epsilon, x}(t), \mathfrak{r}^{i}(t))-\phi({\bf u}^{\epsilon, w}(t), \mathfrak{r}^{i}(t))\Big|\rightarrow 0 
\end{align*}
as $t\rightarrow\infty$ for almost all $\omega\in\Omega$. 
Therefore, 
\begin{align*}
\frac{1}{T}\int^{T}_{0}\Big|\phi({\bf u}^{\epsilon, x}(t), \mathfrak{r}^{i}(t))-\phi({\bf u}^{\epsilon, w}(t), \mathfrak{r}^{i}(t))\Big|dt\rightarrow 0
\end{align*}
as $T\rightarrow\infty$. Then it follows from Lebesgue Dominated Convergence Theorem that \eqref{4.9 MS paper}$\rightarrow 0$ as $T\rightarrow \infty$. This implies
\begin{align*}
\int_{H\times\mathcal{S}}\phi(x, i)d\lambda^{\epsilon}_{1}=\int_{H\times\mathcal{S}}\phi(x, i)d\lambda^{\epsilon}_{2}
\end{align*}
for all $\phi(\cdot, j)\in C_{b}(H)$ for all $j\in \mathcal{S}$. So that $\lambda^{\epsilon}_{1}=\lambda^{\epsilon}_{2}$.

Finally, the ergodicity of $\lambda^{\epsilon}$ follows from \cite[Theorem 3.2.6]{DZ} since $\lambda^{\epsilon}$ is the unique stationary measure to $\mathcal{Q}^{\epsilon}_{t}$.
\end{proof}

\section{The Stochastic Navier-Stokes Equations with Markov Switching}\label{original}
Let ${\bf u}_{0}$ be an $H$-valued random variable with $\mathbb{E}(|{\bf u}_{0}|^{3})<\infty$ and ${\bf f}$ be a $V'$-valued function. Recalling the mollifier operator $k_{\epsilon}\cdot$ defined in \eqref{k_{epsilon}}, we write ${\bf u}^{\epsilon}_{0}=k_{\epsilon}{\bf u}_{0}$. Then, by Theorem \ref{existence theorem}, there is a unique strong solution $({\bf u}^{\epsilon}, \mathfrak{r})$ to the regularized equation \eqref{equation 1 MS paper} for given initial data $({\bf u}^{\epsilon}_{0}, \mathfrak{r}_{0})$ and external force $\bf f$. From our earlier work \cite{SNSEs markov}, there is a subsequence $\epsilon_{k}$ of $\epsilon$ such that ${\bf u}^{\epsilon_{k}}\rightarrow{\bf u}$ weakly in the path space $(\Omega^{\dagger}, \tau^{\dagger})$ (see equation \eqref{the path space}), as $k\rightarrow\infty$, and the function ${\bf u}$ is a solution of the three-dimensional stochastic Navier-Stokes equation with Markov switching \eqref{equation 0 MS paper}.

Fix $x\in H$ and $i\in\mathcal{S}$ as the initial condition of ${\bf u}(t)$ and $\mathfrak{r}(t)$, respectively. For all $B\in\mathcal{B}(H)$ and $j\in\mathcal{S}$, let $\mathcal{Q}_{t}(x, i; B, j)$ denote the transition probability of $({\bf u}, \mathfrak{r})$:
\begin{align}\label{trans prob of u}
\mathcal{Q}_{t}(x, i; B, j)=P(({\bf u}(t), \mathfrak{r}(t))\in (B, j)| ({\bf u}(0), \mathfrak{r}(0))=(x, i)).
\end{align}
Let $\phi:H\times \mathcal{S}\rightarrow\mathbb{R}$ be a bounded continuous function. Define
\begin{align}\label{P_t 0 MS paper}
(\mathcal{Q}_{t}\phi)(x, i):&=\sum_{k=1}^{m}\int_{H}\phi(z, k)\mathcal{Q}_{t}(x, i; dz, k)
=\mathbb{E}^{x, i}(\phi({\bf u}(t), \mathfrak{r}(t)));
\end{align}
a similar argument as in \eqref{showing Q is stochastic conti} shows that $\mathcal{Q}_{t}$ is also stochastic continuous. 

We say $\lambda$ is a stationary measure if
\begin{align}\label{stationary of u}
\sum_{j=1}^{m}\int_{H}(\mathcal{Q}_{t}\phi)(y, j)\lambda(dy, j)=\sum_{j=1}^{m}\int_{H}\phi(y, j)\lambda(dy, j)
\end{align}
for all  $t\geq 0$ and $\phi(\cdot, j)\in C_{b}(H)$ for all $j\in \mathcal{S}$.

A detailed reader may observe that the indices $\{n_{\alpha}\}_{\alpha\in\mathbb{N}}$ in \eqref{lambda epsilon n to lambda epsilon} may depends on $\epsilon$. In the following lemma, we will construct a sequence 
$\{n_{\alpha^{\prime}}\}_{\alpha^{\prime}\in\mathbb{N}}$ such that the convergence
$\lambda^{\epsilon_{k}}_{n_{\alpha^{\prime}}}\rightarrow\lambda^{\epsilon_{k}}$ , as $\alpha^{\prime}\rightarrow\infty$, is valid for every $k$.
\begin{lemma}\label{n alpha prime}
For the probability measure $\lambda^{\epsilon}_{n}$ defined in \eqref{lambda epsilon n}, there exists a sequence $\{n_{\alpha^{\prime}}\}_{\alpha^{\prime}\in\mathbb{N}}$ such that $\lambda^{\epsilon_{k}}_{n_{\alpha^{\prime}}}\rightarrow\lambda^{\epsilon_{k}}$ as $\alpha^{\prime}\rightarrow\infty$ for every $k$.
\end{lemma}

\begin{proof}
Consider $\{\lambda^{\epsilon_{1}}_{n}\}_{n\in\mathbb{N}}$. By \eqref{lambda epsilon n to lambda epsilon}, we list the elements of the convergent subsequence 
$\{\lambda^{\epsilon_{1}}_{n_{\alpha_{1}}, \ell}\}_{\ell\in\mathbb{N}}$ and its limit $\lambda^{\epsilon_{1}}$ as follows:
\begin{align*}
\lambda^{\epsilon_{1}}_{n_{\alpha_{1}}, 1}, \lambda^{\epsilon_{1}}_{n_{\alpha_{1}}, 2}, \cdots, \lambda^{\epsilon_{1}}_{n_{\alpha_{1}}, v}, \cdots, \lambda^{\epsilon_{1}}.
\end{align*}
For the sequence $\{\lambda^{\epsilon_{2}}_{n_{\alpha_{1}}, \ell}\}_{\ell\in\mathbb{N}}$, by the same argument, we see that it admits a convergent subsequence $\{\lambda^{\epsilon_{2}}_{n_{\alpha_{2}}, \ell}\}_{\ell\in\mathbb{N}}$ and a limit $\lambda^{\epsilon_{2}}$:
\begin{align*}
\lambda^{\epsilon_{2}}_{n_{\alpha_{2}}, 1}, \lambda^{\epsilon_{2}}_{n_{\alpha_{2}}, 2}, \cdots, \lambda^{\epsilon_{2}}_{n_{\alpha_{2}}, v}, \cdots, \lambda^{\epsilon_{2}}.
\end{align*}
Repeating this procedure, we will have the following list:
\begin{align}\label{diagonalization 1}
\begin{split}
&\lambda^{\epsilon_{1}}_{n_{\alpha_{1}}, 1}, \lambda^{\epsilon_{1}}_{n_{\alpha_{1}}, 2}, \cdots, \lambda^{\epsilon_{1}}_{n_{\alpha_{1}}, v}, \cdots, \lambda^{\epsilon_{1}}\\
&\lambda^{\epsilon_{2}}_{n_{\alpha_{2}}, 1}, \lambda^{\epsilon_{2}}_{n_{\alpha_{2}}, 2}, \cdots, \lambda^{\epsilon_{2}}_{n_{\alpha_{2}}, v}, \cdots, \lambda^{\epsilon_{2}}\\
&\vdots\\
& \lambda^{\epsilon_{v}}_{n_{\alpha_{v}}, 1}, \lambda^{\epsilon_{v}}_{n_{\alpha_{v}}, 2}, \cdots, \lambda^{\epsilon_{v}}_{n_{\alpha_{v}}, v}, \cdots, \lambda^{\epsilon_{v}}\\
&\vdots
\end{split}
\end{align}
For the subindices of the diagonal elements, $\{n_{\alpha_{v}}, v\}_{v\in\mathbb{N}}$, we define, for all $v\in\mathbb{N}$,
\begin{align*}
n_{v}:=(n_{\alpha_{v}}, v)
\end{align*}
Now we would like to show that $\lambda^{\epsilon_{k}}_{n_{v}}\rightarrow\lambda^{\epsilon_{k}}$ weakly, as $v\rightarrow\infty$, for every $k$. For arbitrary $k\in\mathbb{N}$, we have
\begin{align}\label{indices trouble}
&\Big|\sum_{j=1}^{m}\int_{H}\phi(y, j)\lambda^{\epsilon_{k}}_{n_{v}}(dy, j)-\sum_{j=1}^{m}\int_{H}\phi(y, j)\lambda^{\epsilon_{k}}(dy, j)\Big|
\\&\leq\Big|\sum_{j=1}^{m}\int_{H}\phi(y, j)\lambda^{\epsilon_{k}}_{n_{v}}(dy, j)-\sum_{j=1}^{m}\int_{H}\phi(y, j)\lambda^{\epsilon_{k}}_{n_{\alpha_{k}}, v}(dy, j)\Big|\nonumber
\\&\quad+
\Big|\sum_{j=1}^{m}\int_{H}\phi(y, j)\lambda^{\epsilon_{k}}_{n_{\alpha_{k}}, v}(dy, j)-\sum_{j=1}^{m}\int_{H}\phi(y, j)\lambda^{\epsilon_{k}}(dy, j)\Big|\nonumber.
\end{align}
The second term on the right of \eqref{indices trouble} will be arbitrarily small if $v$ is sufficiently large (look at the $k$-th row of \eqref{diagonalization 1}). If $v>k$, $n_{v}:=(n_{\alpha_{v}}, v)$ is an element in a subsequence of the original sequence $\{n_{\alpha_{k}}, \ell\}_{\ell\in\mathbb{N}}$; therefore, viewing $n_{v}$ as a member of  $\{n_{\alpha_{k}}, \ell\}_{\ell\in\mathbb{N}}$, by the Cauchy criterion, we conclude that the first term on the right of \eqref{indices trouble} will be arbitrarily small if $v$ is sufficiently large.

Calling $n_{v}$ as $n_{\alpha^{\prime}}$, we complete the proof.
\end{proof}

Now we are in a position to prove the existence of a stationary measure to the original system.

\begin{theorem}\label{existence of stationary 0 MS paper}
Assume that $\mathbb{E}(|{\bf u}_{0}|^{3})<\infty$ and $\|{\bf f}\|^{2}_{V'}=F>0$. Then there exists a stationary measure $\lambda$, with support in $V\times\mathcal{S}$ and a finite second moment, for a solution ${\bf u}$ of the three-dimensional stochastic Navier-Stokes equation with Markov switching \eqref{equation 0 MS paper}. Moreover, if we denote by $\{\lambda^{\epsilon}\}_{\epsilon>0}$ the family of ergodic measures obtained in Theorem \ref{existence of stationary MS paper}, then there exists a sequence $\epsilon_{\ell}$ decreasing to 0 such that $\lambda^{\epsilon_{\ell}}\rightarrow\lambda$ weakly as $\ell\rightarrow\infty$.
\end{theorem}
The proof of this theorem consists of four parts:
\begin{enumerate}
\item The existence of a stationary measure $\lambda$ to the original system \eqref{equation 0 MS paper}.
\item Implement a diagonalization argument for $\lambda^{\epsilon}_{n}$ defined in \eqref{lambda n MS paper}. Then we obtain a convergent subsequence $\{\lambda^{\epsilon_{\ell}}_{n_{\ell}}\}_{\ell\in\mathbb{N}}$ and a limit $\lambda^{\ast}$.
\item Identification of the limit: $\lambda=\lambda^{\ast}$.
\item Relate $\lambda^{\epsilon}$ to $\lambda$, i.e., show that $\lambda^{\epsilon_{\ell}}\rightarrow\lambda$ weakly as $\ell\rightarrow\infty$.
\end{enumerate}
We shall begin by showing the existence of a stationary measure.

\begin{proof}[Step 1: The existence of a stationary measure]
It follows from \cite{SNSEs markov} that ${\bf u}^{\epsilon}$ converges along a sequence to ${\bf u}$, as $\epsilon\rightarrow 0$, in the path space $(\Omega^{\dagger}, \tau^{\dagger})$.  Hence, there exist a sequence $\{\epsilon_{k}\}_{k\in\mathbb{N}}$ such that ${\bf u}^{\epsilon_{k}}\rightarrow{\bf u}$, as $k\rightarrow \infty$, in $(\Omega^{\dagger}, \tau^{\dagger})$. Therefore, in particular, ${\bf u}^{\epsilon_{k}}\rightarrow{\bf u}$ weakly in $L^{2}(\Omega; L^{2}(0, T; V))$ as $k\rightarrow\infty$. It follows from the lower semi-continuity of the norm (with respect to weak topology) and \eqref{4.5 MS paper} that
\begin{align}
\frac{\nu}{t}\mathbb{E}\Big(\int^{t}_{0}\|{\bf u}(s)\|^{2}ds\Big)\leq\liminf_{k\rightarrow \infty}\frac{\nu}{t}\mathbb{E}\Big(\int^{t}_{0}\|{\bf u}^{\epsilon_{k}}(s)\|^{2}ds\Big)\leq C \label{L^2 V bound of u}
\end{align}
for $t>1$ and $C$ is an appropriate constant independent of $t$ and $\epsilon$. Hence by the Chebyshev inequality
\begin{align}
\lim_{N\rightarrow\infty}\sup_{t>1}\frac{1}{t}\int^{t}_{0}P(\|{\bf u}(s)\|>N)ds=0\label{4.6 0 MS paper}
\end{align}
follows.

Let $\{t_{n_{\alpha^{\prime}}}\}$ be the increasing sequence obtained in Lemma \ref{n alpha prime}. Define probability measures $\lambda_{n_{\alpha^{\prime}}}$ as follows (with initial condition $({\bf u}(0), \mathfrak{r}(0))=(x, i)$):
\begin{align}
\lambda_{n_{\alpha^{\prime}}}(B, j):= \frac{1}{t_{n_{\alpha^{\prime}}}}\int^{t_{n_{\alpha^{\prime}}}}_{0}\mathcal{Q}_{s}(x, i; B, j)ds\label{lambda n 0 MS paper}
\end{align}
for all $B\in\mathcal{B}(H)$ and $j\in\mathcal{S}$. Then it follows from \eqref{4.6 0 MS paper} that $\{\lambda_{n_{\alpha^{\prime}}}\}_{\alpha^{\prime}\in\mathbb{N}}$ is tight in the space of probability measures on $(H\times\mathcal{S}, \mathcal{B}(H\times\mathcal{S}))$ equipped with weak topology. By the Prokhorov's theorem, there exists a subsequence $\{\lambda_{n_{\alpha^{\prime}_{\beta}}}\}_{\beta\in\mathbb{N}}$ such that 
\begin{align}\label{lambda n to lambda 0}
\lambda_{n_{\alpha^{\prime}_{\beta}}}\quad\text{converges in law to a limit}\quad\lambda\quad\mbox{as}\quad\beta\rightarrow\infty.
\end{align}
From now on, we shall denote $n_{_{\beta}}=n_{\alpha^{\prime}_{\beta}}$ for simplicity.

For $\phi(\cdot, j)\in C_{b}(H)$ for all $j\in \mathcal{S}$, we have
\begin{align*}
&\sum_{j=1}^{m}\int_{H}(\mathcal{Q}_{t}\phi)(y, j)\lambda(dy, j)
=\lim_{{\beta}\rightarrow\infty}\sum_{j=1}^{m}\int_{H}(\mathcal{Q}_{t}\phi)(y, j)\lambda_{n_{\beta}}(dy, j)
\\&=\lim_{\beta\rightarrow\infty}\sum_{j=1}^{m}\int_{H}(\mathcal{Q}_{t}\phi)(y, j)\frac{1}{t_{n_{_\beta}}}\int^{t_{n_{_\beta}}}_{0}\mathcal{Q}_{s}(x, i; dy, j)ds.
\end{align*}
Using \eqref{P_t 0 MS paper} in above, we further have
\begin{align}
&\sum_{j=1}^{m}\int_{H}(\mathcal{Q}_{t}\phi)(y, j)\lambda(dy, j)\label{stationary lambda 0 MS paper}
\\&=\lim_{\beta\rightarrow\infty}\sum_{j=1}^{m}\sum_{k=1}^{m}\int_{H}\int_{H}\phi(z, k)\mathcal{Q}_{t}(y, j; dz, k)\frac{1}{t_{n_{_\beta}}}\int^{t_{n_{_\beta}}}_{0}\mathcal{Q}_{s}(x, i; dy, j)ds.\nonumber
\end{align}

Recalling the Chapman-Kolmogorov equation, i.e., 
\begin{align*}
\mathcal{Q}_{s+t}(x, i; dz, k)=\sum_{j=1}^{m}\int_{H}\mathcal{Q}_{s}(x, i; dy, j)\mathcal{Q}_{t}(y, j; dz, k)
\end{align*}
and plugging it into \eqref{stationary lambda 0 MS paper}, we have
\begin{align*}
\sum_{j=1}^{m}\int_{H}(\mathcal{Q}_{t}\phi)(y, j)\lambda(dy, j)=\lim_{\beta\rightarrow\infty}\sum_{k=1}^{m}\int_{H}\phi(z, k)\frac{1}{t_{n_{_\beta}}}\int^{t_{n_{_\beta}}}_{0}\mathcal{Q}_{s+t}(x, i; dz, k)ds.
\end{align*}

Note that
\begin{align*}
&\lim_{\beta\rightarrow\infty}\frac{1}{t_{n_{_\beta}}}\int^{t_{n_{_\beta}}}_{0}\mathcal{Q}_{s+t}(x, i; dz, k)ds
=\lim_{\beta\rightarrow\infty}\frac{1}{t_{n_{_\beta}}}\int^{t_{n_{_\beta}}+t}_{t}\mathcal{Q}_{u}(x, i; dz, k)du
\\&=\lim_{\beta\rightarrow\infty}\frac{1}{t_{n_{_\beta}}}\Big(\int^{t_{n_{_\beta}}}_{0}\mathcal{Q}_{u}(x, i; dz, k)du+\int^{t_{n_{_\beta}}+t}_{t_{n_{_\beta}}}\mathcal{Q}_{u}(x, i; dz, k)du
\\&\qquad\qquad\quad-\int^{t}_{0}\mathcal{Q}_{u}(x, i; dz, k)du\Big)
\\&=\lim_{\beta\rightarrow\infty}\frac{1}{t_{n_{_\beta}}}\int^{t_{n_{_\beta}}}_{0}\mathcal{Q}_{u}(x, i; dz, k)du
\end{align*}

Thus, we have
\begin{align*}
&\sum_{j=1}^{m}\int_{H}(\mathcal{Q}_{t}\phi)(y, j)\lambda(dy, j)
=\lim_{\beta\rightarrow\infty}\sum_{k=1}^{m}\int_{H}\phi(z, k)\frac{1}{t_{n_{_\beta}}}\int^{t_{n_{_\beta}}}_{0}\mathcal{Q}_{s+t}(x, i; dz, k)ds
\\&=\lim_{\beta\rightarrow\infty}\sum_{k=1}^{m}\int_{H}\phi(z, k)\frac{1}{t_{n_{_\beta}}}\int^{t_{n_{_\beta}}}_{0}\mathcal{Q}_{u}(x, i; dz, k)du
=\sum_{k=1}^{m}\int_{H}\phi(z, k)\lambda(dz, k).
\end{align*}
Hence, we conclude that $\lambda$ is a stationary measure.

For the second moment of $\lambda$, one employs the lower semi-continuity of the $H$-norm to deduce
\begin{align*}
\sum_{j=1}^{m}\int_{H}(|y|^{2}+j^{2})\lambda(dy, j)\leq\liminf_{\beta\rightarrow\infty}\sum_{j=1}^{m}\int_{H}(|y|^{2}+j^{2})\lambda_{n_{_\beta}}(dy, j).
\end{align*}
Then using \eqref{P_t 0 MS paper} and \eqref{lambda n 0 MS paper} with $\phi(y, j)=|y|^{2}+j^{2}$, one further deduces
\begin{align*}
&\sum_{j=1}^{m}\int_{H}(|y|^{2}+j^{2})\lambda_{n_{_\beta}}(dx, i)
=\sum_{j=1}^{m}\int_{H}\phi(y, j)\frac{1}{n_{_\beta}}\int^{n_{_\beta}}_{0}\mathcal{Q}_{t}(x, i; dy, j)dt
\\&=\frac{1}{n_{_\beta}}\int^{n_{_\beta}}_{0}\sum_{j=1}^{m}\int_{H}\phi(y, j)\mathcal{Q}_{t}(x, i; dy, j)dt
=\frac{1}{n_{_\beta}}\int^{n_{_\beta}}_{0}\mathbb{E}(|{\bf u}^{y}(t)|^{2}+i^{2})dt.
\end{align*}
For the term
\begin{align*}
\frac{1}{n_{_\beta}}\int^{n_{_\beta}}_{0}\mathbb{E}|{\bf u}^{y}(t)|^{2}dt,
\end{align*}
we deduce from $|\cdot|<\|\cdot\|$ and \eqref{L^2 V bound of u} that
\begin{align*}
\frac{1}{n_{_\beta}}\int^{n_{_\beta}}_{0}\mathbb{E}|{\bf u}^{y}(t)|^{2}dt<\frac{1}{t_{n_{_\beta}}}\int^{t_{n_{_\beta}}}_{0}\mathbb{E}\|{\bf u}^{y}(t)\|^{2}dt<C,
\end{align*} 
where $C$ is a generic finite constant independent of $t$ and $\epsilon$. Therefore, the second moment for the measure $\lambda$ exists.

For the support of $\lambda$, let ${\bf u}$ denote the solution of \eqref{equation 0 MS paper} started at ${\bf u}_{0}$ and ${\mathfrak{r}}_{0}$, where the (joint) distribution of $({\bf u}_{0}, \mathfrak{r}_{0})$ is given by $\lambda$. By \eqref{L^2 V bound of u}, it follows that ${\bf u}(s)$ is almost surely $V$-valued for almost all $s$. In particular, $\lambda$ has support in $V\times\mathcal{S}$.
\end{proof}

\begin{proof}[Step 2: A diagonalization argument]

Recall the definition of $\lambda^{\epsilon}_{n}$ from \eqref{lambda n MS paper}. Let $N$ be a positive integer and  $A:=\{v\in V : \|v\|\leq N\}$, which is a bounded set in $V$. Then by the compact embedding $\iota:V\hookrightarrow H$, $A$ is a relatively compact set in $H$. Consider
\begin{align*}
\lambda^{\epsilon}_{n}(A^{c}, j)&=\frac{1}{t_{n}}\int^{t_{n}}_{0}\mathcal{Q}^{\epsilon}_{s}(x, i; A^{c}, j)ds
\\&=\frac{1}{t_{n}}\int^{t_{n}}_{0}P(({\bf u}^{\epsilon}(t), \mathfrak{r}(t))\in (A^{c}, j) | ({\bf u}^{\epsilon}(0), \mathfrak{r}(0))=(x, i))
\\&=\frac{1}{t_{n}}\int^{t_{n}}_{0}P(\|{\bf u}^{\epsilon}(t)\|>N \ \mbox{and} \ \mathfrak{r}(t)=j  | {\bf u}^{\epsilon}(0)=x \  \mbox{and}\ \mathfrak{r}(0)=i).
\end{align*}
By the Chebyshev inequality, we have
\begin{align*}
P(\|{\bf u}^{\epsilon}(t)\|>N \ \mbox{and} \ \mathfrak{r}(t)=j  | {\bf u}^{\epsilon}(0)=x \  \mbox{and}\ \mathfrak{r}(0)=i)
\leq\frac{1}{N^{2}}\mathbb{E}\|{\bf u}^{\epsilon}(t)\|^{2};
\end{align*}
therefore,
\begin{align*}
\lambda^{\epsilon}_{n}(A^{c}, j)=\frac{1}{t_{n}}\int^{t_{n}}_{0}\mathcal{Q}^{\epsilon}_{s}(x, i; A^{c}, j)ds
\leq\frac{1}{N^{2}}\frac{1}{t_{n}}\mathbb{E}\int^{t_{n}}_{0}\|{\bf u}^{\epsilon}(t)\|^{2}ds;
\end{align*}
moreover, by \eqref{L^2 strong MS paper}, we have
\begin{align*}
\lambda^{\epsilon}_{n}(A^{c}, j)&=\frac{1}{t_{n}}\int^{t_{n}}_{0}\mathcal{Q}^{\epsilon}_{s}(x, i; A^{c}, j)ds\leq\frac{1}{N^{2}}\frac{1}{t_{n}}\mathbb{E}\int^{t_{n}}_{0}\|{\bf u}^{\epsilon}(t)\|^{2}ds
\\&\leq\frac{1}{N^{2}}\frac{1}{t_{n}}\Big(\mathbb{E}|{\bf u}_{0}|^{2}+\frac{1}{\nu}t_{n}\|{\bf f}\|^{2}_{V'}+2Kt_{n}\Big)\leq\frac{1}{N^{2}}\sup_{ t_{n}>1}\Big\{\Big(\frac{\mathbb{E}|{\bf u}_{0}|^{2}}{t_{n}}+\frac{F}{\nu}+2K\Big)\Big\}.
\end{align*}
Hence, for a sufficiently large relatively compact set $A$ (i.e., sufficiently large $N$), we have
\begin{align}\label{tight in two variables}
\lambda^{\epsilon}_{n}(A^{c}, j)<\delta
\end{align}
where $\delta$ is independent of $t, \epsilon$, and $n$.

Consider $\lambda^{\epsilon_{k}}_{n_{\beta}}$. For $\{\lambda^{\epsilon_{k}}_{n_{1}}\}_{k\in\mathbb{N}}$, we see from \eqref{tight in two variables} that it is tight; therefore, by Prokhorov's theorem, there exists a  convergent subsequence $\{\lambda^{\epsilon_{k_{1}}}_{n_{1}}\}_{k_{1}\in\mathbb{N}}$, which we list below as follows:
\begin{align*}
\lambda^{\epsilon_{k_{1}, 1}}_{n_{1}}, \lambda^{\epsilon_{k_{1}, 2}}_{n_{1}}, \cdots\lambda^{\epsilon_{k_{1}, v}}_{n_{1}}, \cdots.
\end{align*}
For the sequence $\{\lambda^{\epsilon_{k_{1}}}_{n_{2}}\}_{k_{1}\in\mathbb{N}}$, by the same argument, we see that it admits a convergent subsequence $\{\lambda^{\epsilon_{k_{2}}}_{n_{2}}\}_{k_{2}\in\mathbb{N}}$:
\begin{align*}
\lambda^{\epsilon_{k_{2}, 1}}_{n_{2}}, \lambda^{\epsilon_{k_{2}, 2}}_{n_{2}}, \cdots,\lambda^{\epsilon_{k_{2}, v}}_{n_{2}}, \cdots.
\end{align*}
Repeating this procedure, we will have the following list:
\begin{align*}
&\lambda^{\epsilon_{k_{1}, 1}}_{n_{1}}, \lambda^{\epsilon_{k_{1}, 2}}_{n_{1}}, \cdots\lambda^{\epsilon_{k_{1}, v}}_{n_{1}}, \cdots\\
&\lambda^{\epsilon_{k_{2}, 1}}_{n_{2}}, \lambda^{\epsilon_{k_{2}, 2}}_{n_{2}}, \cdots,\lambda^{\epsilon_{k_{2}, v}}_{n_{2}}, \cdots\\
&\vdots\\
&\lambda^{\epsilon_{k_{v}, 1}}_{n_{v}}, \lambda^{\epsilon_{k_{v}, 2}}_{n_{v}}, \cdots,\lambda^{\epsilon_{k_{v}, v}}_{n_{v}},\cdots\\
&\vdots
\end{align*}
Collecting the diagonal elements $\{\lambda^{\epsilon_{k_{v}}, v}_{n_{v}}\}_{v\in\mathbb{N}}$, we see again from \eqref{tight in two variables} that it is tight; therefore, there exists a subsequence $\{\lambda^{\epsilon_{\ell}}_{n_{\ell}}\}_{\ell\in\mathbb{N}}$ such that
\begin{align}\label{lambda star}
\lambda^{\epsilon_{\ell}}_{n_{\ell}}\quad\text{converges in law to a limit}\quad\lambda^{\ast}\quad\mbox{as}\quad\ell\rightarrow\infty.
\end{align}
\end{proof}

\begin{proof}[Step 3: Identification of the limit]
Now we claim that $\lambda^{\ast}=\lambda$. Consider, for all $\phi(\cdot, j)\in C_{b}(H)$ for all $j\in\mathcal{S}$,
\begin{align}
&\Big|\sum_{j=1}^{m}\int_{H}\phi(y, j)\lambda(dy, j)-\sum_{j=1}^{m}\int_{H}\phi(y, j)\lambda^{\ast}(dy, j)\Big|\label{the inequality}
\\&\leq\Big|\sum_{j=1}^{m}\int_{H}\phi(y, j)\lambda(dy, j)-\sum_{j=1}^{m}\int_{H}\phi(y, j)\lambda_{n_{\ell}}(dy, j)\Big|\nonumber
\\&\quad+\Big|\sum_{j=1}^{m}\int_{H}\phi(y, j)\lambda_{n_{\ell}}(dy, j)-\sum_{j=1}^{m}\int_{H}\phi(y, j)\lambda^{\epsilon_{\ell}}_{n_{\ell}}(dy, j)\Big|\nonumber
\\&\quad+\Big|\sum_{j=1}^{m}\int_{H}\phi(y, j)\lambda^{\epsilon_{\ell}}_{n_{\ell}}(dy, j)-\sum_{j=1}^{m}\int_{H}\phi(y, j)\lambda^{\ast}(dy, j)\Big|.\nonumber
\end{align}
For any $\eta>0$, there exist $N_{1}\in\mathbb{N}$ such that 
\begin{align}\label{1 of 1/3}
\Big|\sum_{j=1}^{m}\int_{H}\phi(y, j)\lambda(dy, j)-\sum_{j=1}^{m}\int_{H}\phi(y, j)\lambda_{n_{\ell}}(dy, j)\Big|<\frac{\eta}{3}
\end{align}
if $\ell>N_{1}$ by \eqref{lambda n to lambda 0} and the fact that $\{n_{\ell}\}_{\ell\in\mathbb{N}}$ is a subsequence of $\{n_{_{\beta}}\}_{\beta\in\mathbb{N}}$; there exists $N_{3}\in\mathbb{N}$ such that
\begin{align}\label{3 of 1/3}
\Big|\sum_{j=1}^{m}\int_{H}\phi(y, j)\lambda^{\epsilon_{\ell}}_{n_{\ell}}(dy, j)-\sum_{j=1}^{m}\int_{H}\phi(y, j)\lambda^{\ast}(dy, j)\Big|<\frac{\eta}{3}
\end{align}
if $\ell>N_{3}$ by \eqref{lambda star}. 

For the second term in \eqref{the inequality}, we use \eqref{P_t MS paper} and \eqref{P_t 0 MS paper} to simplify:
\begin{align*}
&\Big|\sum_{j=1}^{m}\int_{H}\phi(y, j)\lambda_{n_{\ell}}(dy, j)-\sum_{j=1}^{m}\int_{H}\phi(y, j)\lambda^{\epsilon_{\ell}}_{n_{\ell}}(dy, j)\Big|
\\&=\Big|\frac{1}{t_{n_{\ell}}}\int^{t_{n_{\ell}}}_{0}\mathbb{E}^{x, i}\phi({\bf u}(s), \mathfrak{r}(s))ds-\frac{1}{t_{n_{\ell}}}\int^{t_{n_{\ell}}}_{0}\mathbb{E}^{x, i}\phi({\bf u}^{\epsilon_{\ell}}(s), \mathfrak{r}(s))ds\Big|\nonumber.
\end{align*}
Since ${\bf u}^{\epsilon_{\ell}}\rightarrow{\bf u}$ in law, there exists $N_{2}\in\mathbb{N}$ such that
\begin{align}\label{2 of 1/3}
&\Big|\frac{1}{t_{n_{\ell}}}\int^{t_{n_{\ell}}}_{0}\mathbb{E}^{x, i}\phi({\bf u}(s), \mathfrak{r}(s))ds-\frac{1}{t_{n_{\ell}}}\int^{t_{n_{\ell}}}_{0}\mathbb{E}^{x, i}\phi({\bf u}^{\epsilon_{\ell}}(s), \mathfrak{r}(s))ds\Big|
\\&\leq\frac{1}{t_{n_{\ell}}}\int^{t_{n_{\ell}}}_{0}|\mathbb{E}^{x, i}\phi({\bf u}(s), \mathfrak{r}(s))-\mathbb{E}^{x, i}\phi({\bf u}^{\epsilon_{\ell}}(s), \mathfrak{r}(s))|ds
\leq\frac{1}{t_{n_{\ell}}}\int^{t_{n_{\ell}}}_{0}\frac{\eta}{3}ds=\frac{\eta}{3}\nonumber
\end{align}
if $\ell> N_{2}$. Taking $N=\max_{1\leq j\leq 3}{N_{j}}$, we see from \eqref{1 of 1/3}, \eqref{3 of 1/3}, and \eqref{2 of 1/3} that
\begin{align*}
&\Big|\sum_{j=1}^{m}\int_{H}\phi(y, j)\lambda(dy, j)-\sum_{j=1}^{m}\int_{H}\phi(y, j)\lambda^{\ast}(dy, j)\Big|<\eta
\end{align*}
when $\ell>N$. Thus, we conclude that $\lambda=\lambda^{\ast}$.
\end{proof}

\begin{proof}[Step 4: Relate $\lambda^{\epsilon}$ to $\lambda$]

Since $\lambda=\lambda^{\ast}$, \eqref{lambda star} implies that
\begin{align}\label{1 of lambda epsilon to lambda}
\lambda^{\epsilon_{\ell}}_{n_{\ell}}\rightarrow\lambda
\end{align}
weakly as $\ell\rightarrow\infty$; moreover, by Lemma \ref{n alpha prime}, we see that, for any fixed $\epsilon_{k}$,
\begin{align}\label{2 of lambda epsilon to lambda}
\lim_{\ell\rightarrow\infty}\sum_{j=1}^{m}\int_{H}\phi(y, j)\lambda^{\epsilon_{k}}_{n_{\ell}}(dy, j)=\sum_{j=1}^{m}\int_{H}\phi(y, j)\lambda^{\epsilon_{k}}(dy, j).
\end{align}
Thus, by \eqref{1 of lambda epsilon to lambda} and \eqref{2 of lambda epsilon to lambda}, we have
\begin{align*}
\lim_{\ell\rightarrow\infty}\sum_{j=1}^{m}\int_{H}\phi(y, j)\lambda^{\epsilon_{\ell}}(dy, j)=\sum_{j=1}^{m}\int_{H}\phi(y, j)\lambda(dy, j)
\end{align*}
which completes the proof.
\end{proof}

\section*{Acknowledgements}
This work is a part of the PhD thesis of the first author. He thanks Professor Sundar for guidance on this project, and Professor Xiaoliang Wan for financial support from NSF grant DMS-1622026.


\begin{thebibliography}{99}








\bibitem{Albeverio} S. Albeverio and A. B. Cruzeiro:
Global flows with invariant (Gibbs) measures for Euler and Navier-Stokes two-dimensional fluids.
\emph{Comm. Math. Phys.} {\bf 129} (1990), 431–444.


\bibitem{turbulence}\text{B. Birnir}: 
\emph{The Kolmogorov-Obukhov Theory of Turbulence: A Mathematical Theory of Turbulence}. 
SpringerBriefs in Mathematics, Springer, New York, 2013.









\bibitem{Mao Yuan} C. Yuan and X. Mao:
Asymptotic stability in distribution of stochastic differential equations with Markovian switching. 
\emph{Stochastic Process. Appl.} {\bf 103} (2003), 277–291.



\bibitem{DD}G. Da Prato and A. Debussche:
Ergodicity for the 3D stochastic Navier-Stokes equations.
\emph{J. Math. Pures Appl.} {\bf 82} (2003), 877–947.


\bibitem{DZ} G. Da Prato and J. Zabczyk:
\emph{Ergodicity for Infinite Dimensional Systems}.
London Mathematical Society Lecture Note Series, 229. Cambridge University Press, Cambridge, 1996.







\bibitem{Debussche} A. Debussche:
\emph{Ergodicity results for the stochastic Navier-Stokes equations: an introduction.}
Topics in mathematical fluid mechanics, 23–108,
Lecture Notes in Math., 2073, Fond. CIME/CIME Found. Subser., Springer, Heidelberg, 2013.



\bibitem{Evans} L. C. Evans:
\emph{Partial Differential Equations}.
Second edition. Graduate Studies in Mathematics, 19. American Mathematical Society, Providence, RI, 2010.


\bibitem{FMRT} C. Foias, O. Manley, R. Rosa, and R. Temam:
\emph{Navier-Stokes Equations and Turbulence}.
Encyclopedia of Mathematics and its Applications, 83. Cambridge University Press, Cambridge, 2001.


\bibitem{Flandoli} F. Flandoli: 
Dissipativity and invariant measures for stochastic Navier-Stokes equations. 
\emph{Nonlinear Differential Equations Appl.} {\bf 1} (1994), 403-423.


\bibitem{hybrid} G. George Yin and Chao Zhu: 
\emph{Hybrid Switching Diffusions: Properties and Applications}. 
Stochastic Modelling and Applied Probability, 63. Springer, New York, 2010.


\bibitem{Flandoli and Maslowski} F. Flandoli and B. Maslowski: 
Ergodicity for the 2-D Navier-Stokes equation under random perturbations. 
\emph{Comm. Math. Phys.} {\bf172} (1995), 119-141.


\bibitem{FG} F. Flandoli and D. Gatarek:
Martingale and stationary solutions for stochastic Navier-Stokes equations.
\emph{Probab. Theory Related Fields} {\bf 102} (1995), 367–391.


\bibitem{Howes} N. R. Howes:
\emph{Modern Analysis and Topology}.
Universitext. Springer-Verlag, New York, 1995.


\bibitem{SNSEs markov} P.-H. Hsu and P. Sundar:
Three-Dimensional Stochastic Navier-Stokes equations with Markov switching.
https://arxiv.org/abs/2203.14442



\bibitem{I-W} N. Ikeda and S. Watanabe: 
\emph{Stochastic Differential Equations and Diffusion Processes}.
North-Holland Mathematical Library, Second edition, North-Holland Publishing Co., Amsterdan, 1989. 




\bibitem{Leray} J. Leray:
Sur le mouvement d'un liquide visqueux emplissant l'espace. 
\emph{Acta math.} {\bf 63} (1934), 193-248.






\bibitem{MSS} T. M. Mohan, K. Sakthivel, S. S. Sritharan:
Ergodicity for the 3D stochastic Navier-Stokes equations perturbed by Lévy noise. 
\emph{Math. Nachr.} {\bf 292} (2019), 1056–1088.


\bibitem{Mattingly} J. Mattingly: 
Ergodicity of the 2-D Navier-Stokes equations with random forcing and large viscosity. 
\emph{Comm. Math. Phys.} {\bf 206} (1999), 273-288.


\bibitem{Metivier} M. Metivier: 
\emph{Stochastic Partial Differential Equations in Infinite Dimensional Spaces}. 
Quaderni, Scuola Normale Superiore, Pisa, 1988.


\bibitem{Mao} X. Mao:
Stability of stochastic differential equations with Markovian switching. 
\emph{Stochastic Process. Appl.} {\bf 79} (1999), 45–67.


\bibitem{Mao book} X. Mao and C. Yuan:
\emph{Stochastic Differential Equations with Markovian Switching}. 
Imperial College Press, London, 2006.


\bibitem{MYY} X. Mao, C. Yuan, and G. Yin:
Numerical method for stationary distribution of stochastic differential equations with Markovian switching.
\emph{J. Comput. Appl. Math.} {\bf 174} (2005), 1-27.


\bibitem{O-P} W. S. O{\.z}a{\'n}ski and B. C. Pooley: 
Leray's fundamental work on the Naver-Stokes equations: a modern review of ``Sur le mouvement d'un liquide visqueux emplissant l'epsace''.
\emph{Partial Differential Equations in Fluid Dynamics}, London Math. Soc. Lecture Note Ser., 452, Cambridge Univ. Press, 2018.


\bibitem{concise} C. Pr\'ev\^ot and M. R\"ockner: 
\emph{A Concise Course on Stochastic Partial Differential Equations}. 
Lecture Notes in Mathematics, 1905, Springer, Berlin, 2007.




\bibitem{Sundar LIL} P. Sundar: 
Law of iterated logarithm for solutions of stochastic differential equations. 
\emph{Stoch. Anal. Appl.} {\bf 5} (1987), 311-321.


\bibitem{Sundar measure} P. Sundar: 
A note on the solution of a stochastic partial differential equation. 
\emph{Stochastic Anal. Appl.} {\bf22} (2004), 923–938.




\bibitem{Skorohod} A. V. Skorohod: 
\emph{Asymptotic Methods in the Theory of Stochastic Differential Equations}. 
Translations of Mathematical Monographs, 78. American Mathematical Society, Providence, RI, 1989.


\bibitem{sohr} H. Sohr: 
\emph{The Navier-Stokes Equations. An Elementary Functional Analytic Approach}.  
Modern Birkh\"auser Classics. Birkh\"auser/Springer Basel AG, Basel, 2001.




\bibitem{Temam} R. Temam:
\emph{Navier-Stokes Equations. Theory and numerical analysis}.
North-Holland Publishing Co., Amsterdam, 1984.


\bibitem {Vishik} M. J. Vishik and A. V. Fursikov: 
\emph{Mathematical Problems in Statistical Hydromechanics}. 
Kluwer Academic Publ., Boston, 1988.



\end{thebibliography}
\end{document}